\documentclass[11pt]{article}
\usepackage{latexsym,amssymb,amsmath,amsthm,graphicx,float,bm}

\newcommand{\pd}{\partial}

\newcommand{\R}{\mathbb{R}}
\newcommand{\Z}{\mathbb{Z}}

\newcommand{\<}{\langle}
\renewcommand{\>}{\rangle}

\newcommand{\cl}{\mbox{\scriptsize cl}}

\newcommand{\grad}{\nabla}
\renewcommand{\div}{\mbox{div}}
\newcommand{\la}{\langle}
\newcommand{\ra}{\rangle}
\newcommand{\eps}{\varepsilon}

\newtheorem{theorem}{Theorem}
\newtheorem{lemma}{Lemma}

\title{Discrete Symbol Calculus}
\author{
\begin{tabular}{c}
	Laurent Demanet\\
	 Department of Mathematics \\
	 Stanford University\\
	 450 Serra Mall \\
	 Stanford CA94305
\end{tabular}
\qquad
\begin{tabular}{c}
	Lexing Ying\\
	 Department of Mathematics \\
	 University of Texas at Austin\\
	 1 University Station/C1200\\
	  Austin TX78712 
\end{tabular}
}

\date{June 2008}		

\begin{document}
\maketitle

\begin{abstract}

  This paper deals with efficient numerical representation and
  manipulation of differential and integral operators as symbols in
  phase-space, i.e., functions of space $x$ and frequency $\xi$. The
  symbol smoothness conditions obeyed by many operators in connection
  to smooth linear partial differential equations allow to write
  fast-converging, non-asymptotic expansions in adequate systems of
  rational Chebyshev functions or hierarchical splines. The classical
  results of closedness of such symbol classes under multiplication,
  inversion and taking the square root translate into practical
  iterative algorithms for realizing these operations directly in the
  proposed expansions. Because symbol-based numerical methods handle
  operators and not functions, their complexity depends on the desired
  resolution $N$ very weakly, typically only through $\log N$ factors.
  We present three applications to computational problems related to
  wave propagation: 1) preconditioning the Helmholtz equation, 2)
  decomposing wavefields into one-way components and 3) depth-stepping
  in reflection seismology.


\end{abstract}

{\bf Acknowledgements.}
The first author is partially supported by an NSF grant. The second author is
partially supported by an NSF grant, a Sloan Research
Fellowship, and a startup grant from the University of Texas at Austin.

\section{Introduction}









A typical problem of interest in this paper is the efficient
representation of functions of elliptic linear operators such as
\[
A = I - \div (\alpha(x) \nabla \cdot ),
\]
where $\alpha(x) > c > 0$ is smooth, and $x \in [0,1]^2$ with periodic
boundary conditions. We have in mind the inverse, the square root, and
the exponential of $A$ as important examples of functions of $A$.
While most numerical methods for inverting $A$, say, would manipulate
a right-hand side until convergence, and leverage sparsity or other
properties of $A$ in doing so, the scope of this paper is quite
different. Indeed, we present expansions and iterative algorithms for
manipulating operators as ``symbols'', which make no or little
reference to the functions of $x$ to which these operators may later
be applied.

The central question is of course that of choosing a tractable
representation for differential and integral operators. If a function
$f(x)$ has $N$ degrees of freedom---if for instance it is sampled on
$N$ points---then a direct representation of operators acting on
$f(x)$ would in general require $N^2$ degrees of freedom. There are
many known methods for bringing down this count to $O(N)$ or $O(N \log
N)$ in specific cases, such as leveraging sparsity, computing
convolutions via FFT, low-rank approximations, fast summation methods
\cite{greengard,hackbusch2}, wavelet or x-let expansions \cite{BCR}, partitioned SVD
and H-matrices \cite{borm1,hackbusch1}.

The framework presented in this paper is different, in the sense that
we aim for a complexity essentially independent of $N$, i.e., at most
a low-degree polynomial of $\log N$, for representing and combining
operators belonging to standard classes. Like the methods above, the
operator is ``compressed'' in such a way that applying it to functions
remains simple; it is only for this operation that the complexity
needs to be greater than $N$, in our case $O(N \log N)$.

\subsection{Smooth symbols}\label{sec:closedness}

Let $A$ denote a generic differential or singular integral operator,
with kernel representation
\[
Af(x) = \int k(x,y) f(y) \, dy, \qquad x, y \in \R^d.
\]
Expanding the distributional kernel $k(x,y)$ in some basis would be
cumbersome because of the presence of a singularity along the diagonal
$x = y$. For this reason we choose to consider operators as
\emph{pseudodifferential symbols} $a(x,\xi)$, by considering their
action on the Fourier transform\footnote{Our conventions in this
  paper:
\[
\hat{f}(\xi) = \int e^{-2 \pi i x \cdot \xi} f(x) \, dx. \qquad f(x) = \int e^{2 \pi i x \cdot \xi} \hat{f}(\xi) \, d\xi.
\]
} $\hat{f}(\xi)$ of $f(x)$;
\[
Af(x) = \int e^{2 \pi i x \cdot \xi} a(x,\xi) \hat{f}(\xi) \, d\xi.
\]
One often writes $a(x,D)$ for $A$, where $D = - i \nabla_x/(2 \pi)$. In this
representation, the singularities of $k(x,y)$ are turned into the
oscillating factor $e^{2 \pi i x \cdot \xi}$, which can be discounted
by focusing on the symbol $a(x,\xi)$. The latter is usually smooth,
and in a very peculiar way. It is the special form of the smoothness
estimates for $a$---which we now describe---that guarantees the
efficiency of the discretizations proposed in this paper.

A symbol defined on $\R^d \times \R^d$ is said to be
pseudodifferential of order $m$ (and type $(1,0)$) if it obeys
\begin{equation}\label{eq:type10}
|\pd^\alpha_\xi \pd^\beta_x a(x,\xi)| \leq C_{\alpha\beta} \< \xi \>^{m - |\alpha|}, \qquad \< \xi \> \equiv (1+|\xi|^2)^{1/2},
\end{equation}
for all multi-indices $\alpha, \beta$. This symbol class is denoted
$S^m$; the operator corresponding to some $a$ in this class is denoted
$a(x,D)$, and belongs by definition to
the class $\bm{\Psi}^m$. Manifestly, one power of $\< \xi \>$ is
gained for each differentiation, meaning that the larger $\< \xi \>$,
the smoother $a$. For instance, the symbols of differential operators
are polynomials in $\xi$ and obey (\ref{eq:type10}) when they have
$C^\infty$ coefficients. Large classes of singular integral operators
also have symbols in the class $S^m$ \cite{Ste}.



The standard treatment of pseudodifferential operators makes the
further assumption that some symbols can be represented as
polyhomogeneous series, such as
\begin{equation}\label{eq:classical}
a(x,\xi) \sim \sum_{j \geq 0} a_j \left( x,\mbox{arg } \xi \right) |\xi|^{m-j},
\end{equation}
which defines the ``classical'' symbol class $S^m_{\cl}$ when the
$a_j$ are of class $C^\infty$. Corresponding operators are said to be
in the class $\bm{\Psi}^m_{\cl}$. The series should be understood as
an asymptotic expansion; it converges only when adequate cutoffs
smoothly removing the origin multiply each term. Only then, the series
does not converge to $a(x,\xi)$, but to an approximation that differs
from $a$ by a smoothing remainder $r(x,\xi)$, smoothing in the sense
that $| \pd^\alpha_{\xi} \pd^\beta_x r(x,\xi)| = O(\< \xi
\>^{-\infty})$. For instance, an operator is typically transposed,
inverted, etc. modulo a smoothing remainder \cite{Hor}.

The subclass (\ref{eq:classical}) is central for applications---it is
the cornerstone of theories such as geometrical optics---but the
presence of remainders is a nonessential feature that should be
avoided in the design of efficient numerical methods. The lack of
convergence in (\ref{eq:classical}) may be acceptable in the course of
a mathematical argument, but it takes great additional effort to turn
such series into accurate numerical methods; see \cite{Stolk1} for an
example. In a sense, the goal of this paper is to find adequate
substitutes for (\ref{eq:classical}) that promote asymptotic series
into fast-converging expansions.

There are in general no explicit formulas for the symbols of functions
of an operator. Fortunately, some results in the literature guarantee
exact closedness of the symbol classes (\ref{eq:type10}) or
(\ref{eq:classical}) under inversion and taking the square root,
without smoothing remainders. A symbol $a \in S^m$, or an operator
$a(x,D) \in \bm{\Psi}^m$, is said to be elliptic when there exists $R
> 0$ such that
\[
|a^{-1}(x,\xi)| \leq C \, |\xi|^{-m}, \qquad \mbox{when} \qquad |\xi| \geq R. 
\]

\begin{itemize}
\item It is a basic result that if $A \in \bm{\Psi}^{m_1}, B \in
  \bm{\Psi}^{m_2}$, then $AB \in \bm{\Psi}^{m_1 + m_2}$. See for
  instance Theorem 18.1.8 in \cite{Hor}, Volume 3.
\item It is also a standard fact that if $A \in \bm{\Psi}^{m}$, then
  its adjoint $A^{*} \in \bm{\Psi}^{m}$.
\item If $A \in \bm{\Psi}^m$, and $A$ is elliptic and
  invertible\footnote{In the sense that $A$ is a bijection from
	$H^m(\R^d)$ to $L^2(\R^d)$, hence obeys $\| A f \|_{L^2} \leq C \| f
	\|_{H^m}$. Ellipticity, in the sense in which it is defined for
	symbols, obviously does not imply invertibility.} on $L^2$, then
  $A^{-1} \in \bm{\Psi}^{-m}$. This result was proved by Shubin in 1978
  in \cite{Shu}.
\item For the square root, we also assume ellipticity and
  invertibility. It is furthermore convenient to consider operators on
  compact manifolds, in a natural way through Fourier transforms in each
  coordinate patch, so that they have discrete spectral expansions. A
  square root $A^{1/2}$ of an elliptic operator $A$ with spectral
  expansion $A = \sum_j \lambda_j E_j$, where $E_j$ are the spectral
  projectors, is simply
  \begin{equation}\label{eq:As}
	A^{1/2} = \sum_j \lambda_j^{1/2} E_j,
  \end{equation}
  with of course $(A^{1/2})^2 = A$. In 1967, Seeley \cite{See} studied
  such expressions for elliptic $A \in \bm{\Psi}^m_{\cl}$, in the
  context of a much more general study of complex powers of elliptic
  operators. If in addition $m$ is an even integer, and an adequate
  choice of branch cut is made in the complex plane, then Seeley showed
  that $A^{1/2} \in \bm{\Psi}^{m/2}_{\cl}$; see \cite{Sog} for an
  accessible proof that involves the complex contour ``Dunford''
  integral reformulation of (\ref{eq:As}).
\end{itemize}

We do not know of a corresponding closedness result under taking the
square root, for the non-classical class $\bm{\Psi}^m$. In practice,
we will also manipulate operators that come from PDE on bounded
domains with certain boundary conditions; the extension of the theory
of pseudodifferential operators to bounded domains is a difficult
subject that this paper has no ambition of addressing. Let us also
mention in passing that the exponential of an elliptic,
non-self-adjoint pseudodifferential operator is not in general
pseudodifferential itself.

Numerically, it is easy to check that smoothness of symbols is
remarkably robust under inversion and taking the square root of the
corresponding operators, as the following simple one-dimensional
example shows.

Let $A := 4\pi^2 I - \div (\alpha(x) \grad)$ on the periodic interval $[0,1]$ where
$\alpha(x)$ is a random bandlimited function shown in Figure
\ref{fig:sec11}(a). The symbol of this operator is 
\[
a(x,\xi) = 4 \pi^2 (1 + \alpha(x) |\xi|^2) - 2 \pi i \nabla \alpha(x) \cdot \xi,
\]
which is of order 2. In Figure
\ref{fig:sec11}(b), we plot the values of $a(x,\xi) \la\xi\ra^{-2}$
for $x$ and $\xi$ on a Cartesian grid.

Since $A$ is elliptic and invertible, its inverse $C = A^{-1}$ and
square root $D = A^{1/2}$ are both well defined. Let use $c(x,\xi)$
and $d(x,\xi)$ to denote their symbols. From the above theorems, we
know that the orders of $c(x,\xi)$ and $d(x,\xi)$ are respectively
$-2$ and $1$. We do not believe explicit formulae exist for these
symbols, but the numerical values of $c(x,\xi) \la\xi\ra^{2}$ and
$d(x,\xi) \la\xi\ra^{-1}$ are shown in Figure \ref{fig:sec11}(c) and
(d), respectively. These plots demonstrate regularity of these symbols
in $x$ and in $\xi$; observe in particular the disproportionate
smoothness in $\xi$ for large $|\xi|$, as predicted by the class
estimate (\ref{eq:type10}).

\begin{figure}[h!]
  \begin{center}
    \begin{tabular}{cc}
      \includegraphics[height=2in]{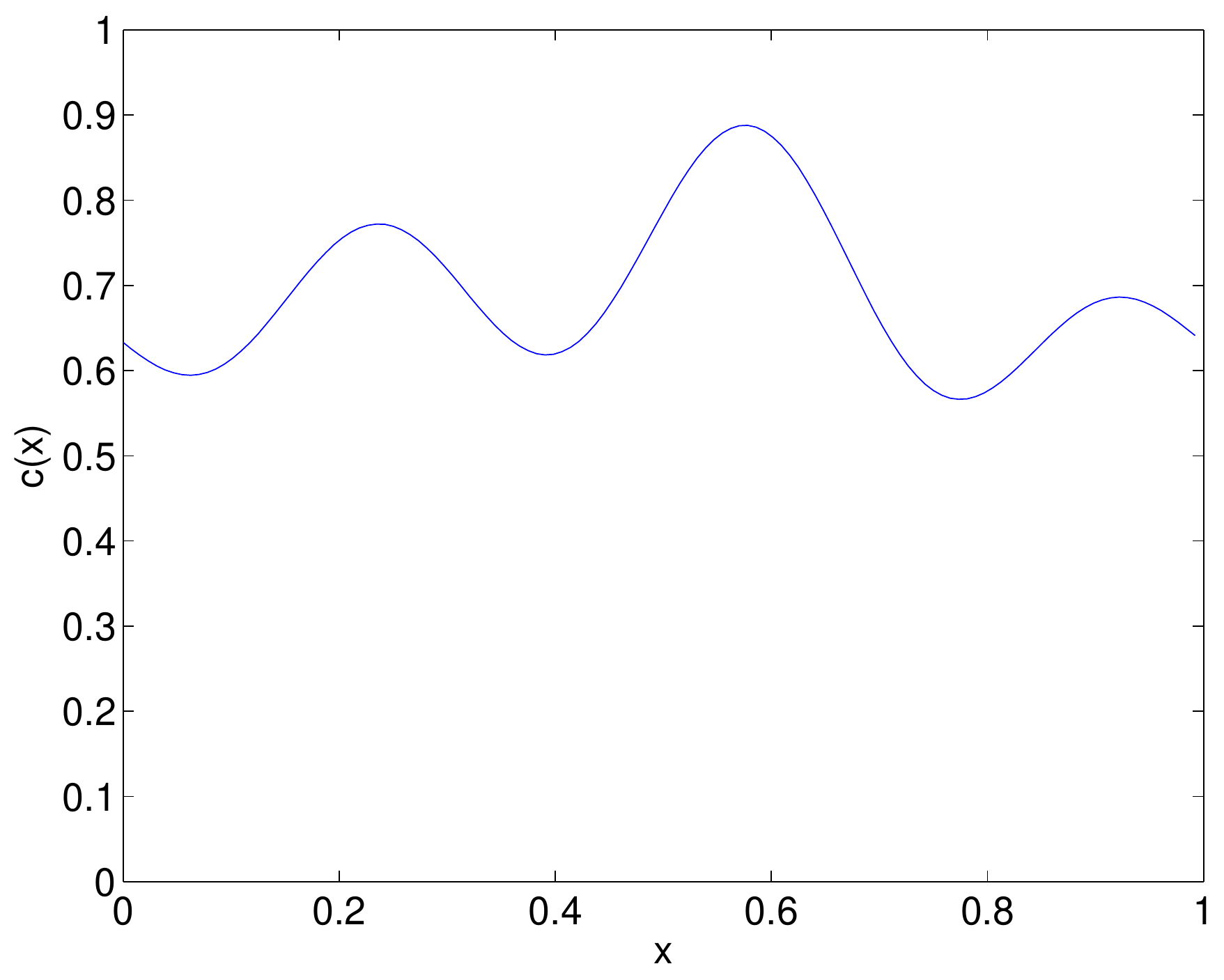} & \includegraphics[height=2in]{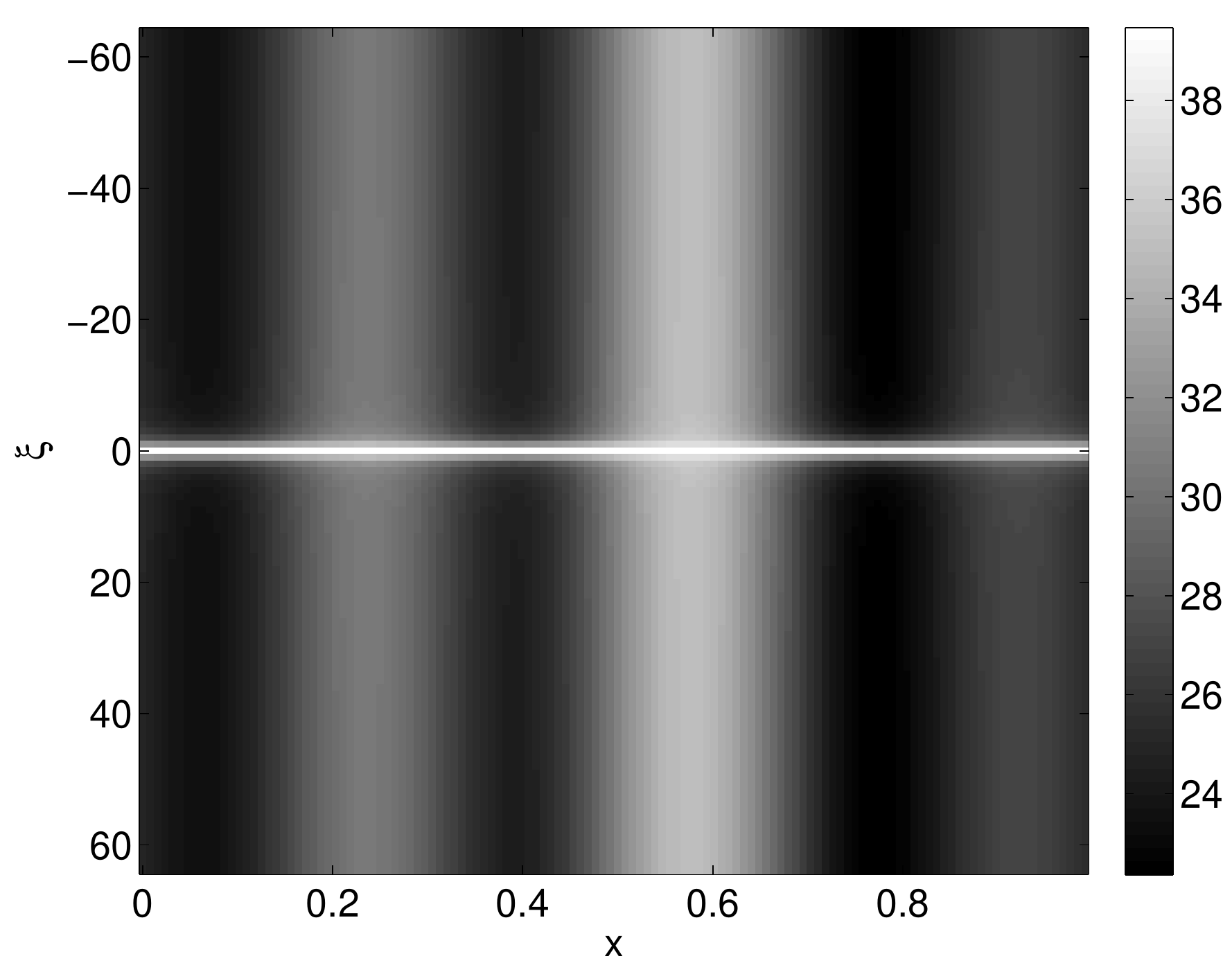}\\
      (a) & (b)\\
      \includegraphics[height=2in]{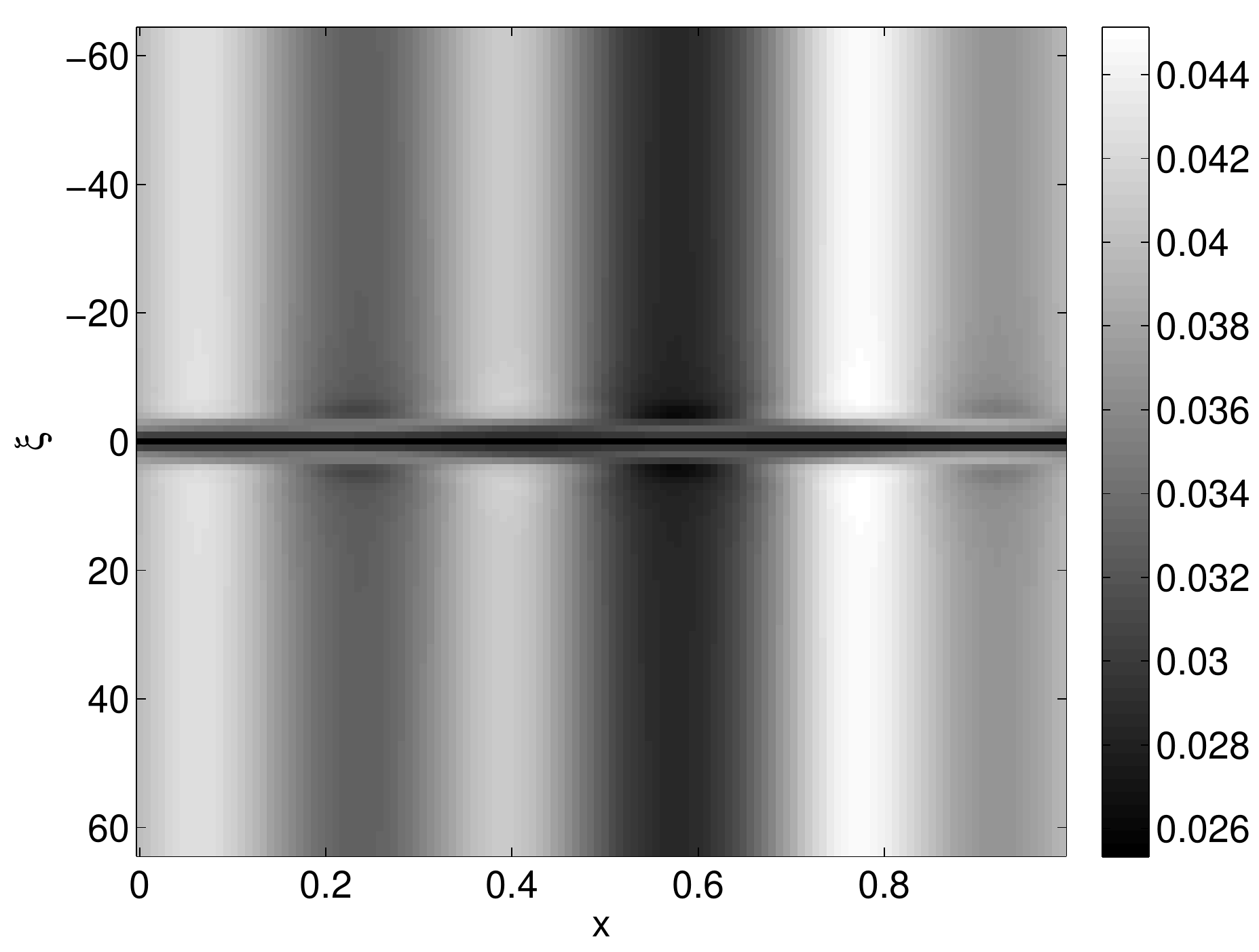} & \includegraphics[height=2in]{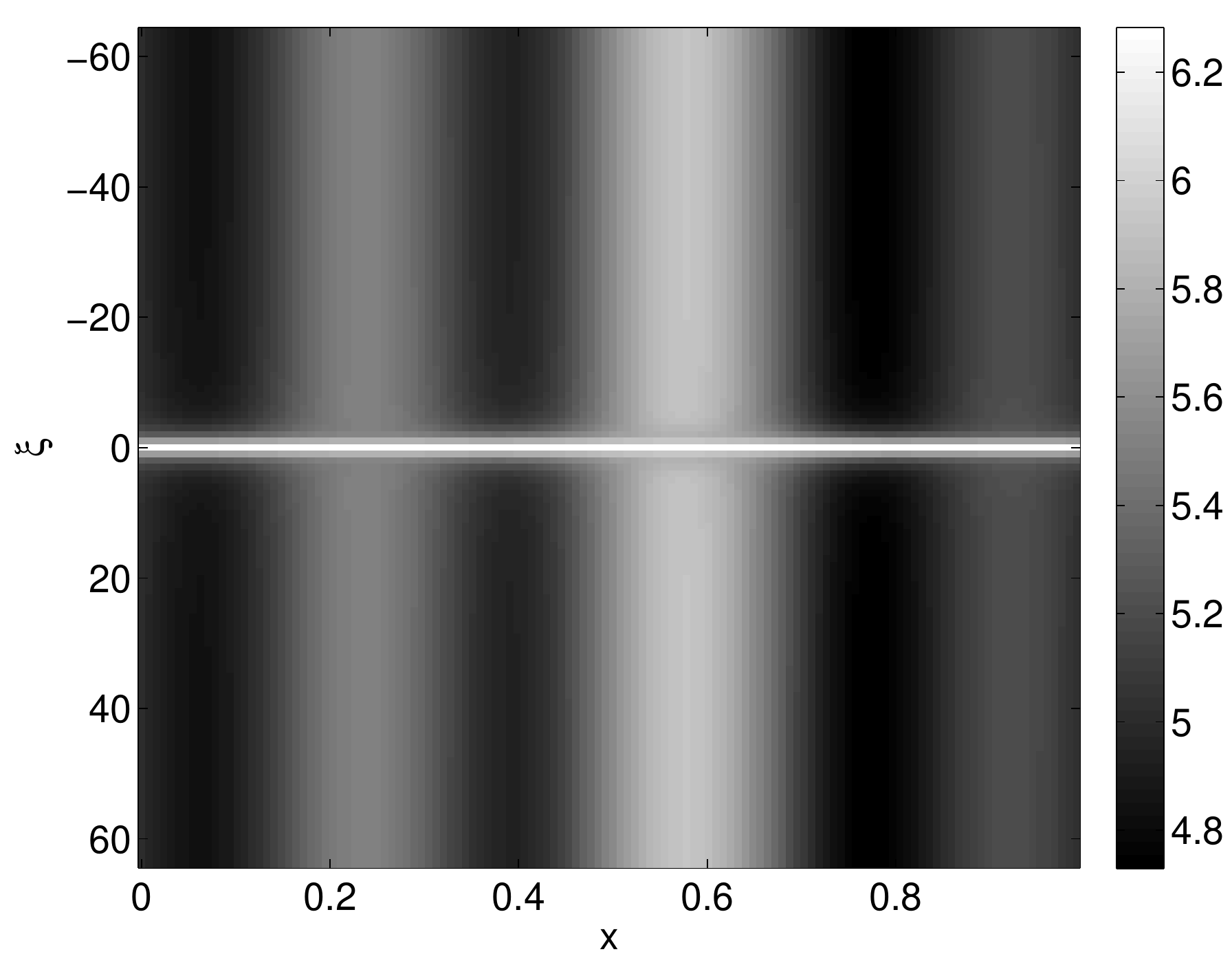}\\
      (c) & (d) 
    \end{tabular}
  \end{center}
  \caption{Smoothness of the symbol in $\xi$.
    (a) The coefficient $\alpha(x)$.
    (b) $a(x,\xi)\la\xi\ra^{-2}$ where $a(x,\xi)$ is the symbol of $A$.
    (c) $c(x,\xi)\la\xi\ra^{ 2}$ where $c(x,\xi)$ is the symbol of $C = A^{-1}$.
    (d) $d(x,\xi)\la\xi\ra^{-1}$ where $d(x,\xi)$ is the symbol of $D = A^{1/2}$.
  }
  \label{fig:sec11}
\end{figure}

\subsection{Symbol expansions}

Figure \ref{fig:sec11} suggests that symbols are not only smooth, but
that they should be highly \emph{separable} in $x$ vs. $\xi$. So we
will use expansions of the form
\begin{equation}\label{eq:sep0}
a(x,\xi) = \sum_{\lambda} a_{\lambda, \mu} e_\lambda(x) g_\mu(\xi) \la \xi\ra^{d_a},
\end{equation}
where $e_\lambda$ and $g_\mu$ are to be determined, and $\la
\xi\ra^{d_a} \equiv (1+|\xi|^2)^{d_a/2}$ encodes the order $d_a$ of
$a(x,\xi)$. This choice is in line with recent observations of Beylkin
and Mohlenkamp \cite{BM} that functions and kernels in high dimensions
should be represented in separated form. In this paper we have chosen
to focus on two-dimensional $x$, i.e. $(x,\xi) \in \R^4$, which is
already considered high-dimensional by the standards of numerical
analysts. For all practical purposes the curse of dimensionality would
prohibit any direct, even coarse sampling in $\R^4$.

The functions $e_\lambda(x)$ and $g_\mu(\xi)$ should be chosen such
that the interaction matrix $a_{\lambda, \mu}$ is as small as possible
after accurate truncation. This choice also depends on the domain over
which the operator is considered. In what follows we will assume that
the $x$-domain is the periodized unit square $[0,1]^2$ in two
dimensions. Accordingly it makes sense to take for $e_\lambda(x)$ the
complex exponentials $e^{2\pi i x \cdot \lambda}$ of a Fourier series.
The choice of $g_\mu(\xi)$ is more delicate, as $x$ and $\xi$ do not
play symmetric roles in the estimate (\ref{eq:type10}). In a nutshell,
we need adequate basis functions for smooth functions on $\R^2$ that
behave like a polynomial of $1/|\xi|$ as $\xi \to \infty$, and
otherwise present smooth angular variations. We present two solutions
in what follows:
\begin{itemize}
\item A \emph{rational Chebyshev interpolant}, where $g_\mu(\xi)$ are
  complex exponentials in angle $\theta = $arg $\xi$, and scaled
  Chebyshev functions in $|\xi|$, where the scaling is an algebraic
  map $s = \frac{|\xi| - L}{|\xi| + L}$. More details in Section
  \ref{sec:Cheb}.
\item A \emph{hierarchical spline interpolant}, where $g_\mu(\xi)$ are
  spline functions with control points placed in a multiscale way in
  the frequency plane, in such a way that they become geometrically
  scarcer as $|\xi| \to \infty$. More details in Section \ref{sec:HS}.
\end{itemize}

Since we are considering $x$ in the periodized square $[0,1]^2$, the
Fourier variable $\xi$ is restricted to having integer values, i.e.,
$\xi \in \Z^2$, and the Fourier transform should be replaced by a
Fourier series. Pseudodifferential operators are then defined through
\begin{equation}\label{eq:psido-sum}
a(x,D) f(x) = \sum_{\xi \in \Z^2} e^{2 \pi i x \cdot \xi} a(x,\xi) \hat{f}(\xi),
\end{equation}
where $\hat{f}(\xi)$ are the Fourier series coefficients of $f$. That
$\xi$ is discrete in this formula should not be a distraction: it is
the smoothness of the underlying function of $\xi \in \R^2$ that
dictates the convergence rate of the proposed expansions.


The following results quantify the performance of the two approximants
introduced above. We refer to an approximant as being truncated to $M$
terms when all but at most $M$ elements are put to zero in the
interaction matrix $a_{\lambda,\mu}$ in (\ref{eq:sep0}).

\begin{theorem}\label{teo:Cheb} \emph{(Rational Chebyshev
    approximants)}.  Assume that $a \in S^m_{\cl}$, that it is
  properly supported, and assume furthermore that the $a_j$ in
  equation (\ref{eq:classical}) have \emph{tempered growth}, in the
  sense that there exists $Q, R > 0$ such that
  \begin{equation}\label{eq:tempered}
    | \pd^\alpha_{\theta} \pd_x^\beta  a_j(x,\theta)| \leq Q_{\alpha, \beta} \cdot R^j.
  \end{equation}
  Denote by $\tilde{a}$ the rational Chebyshev expansion of $a$
  (introduced in Section \ref{sec:Cheb}), properly truncated to $M$
  terms. Call $\tilde{A}$ and $A$ the corresponding pseudodifferential
  operators on $L^2([0,1]^2)$ defined by (\ref{eq:psido-sum}). Then,
  there exists a choice of $M$ obeying
  \[
  M \leq C_n \cdot \eps^{-1/n}, \qquad \forall n > 0,
  \]
  for some $C_n > 0$, such that
  \[
  \| \tilde{A} - A \|_{H^m([0,1]^2) \to L^2([0,1]^2)} \leq \epsilon.
  \]
\end{theorem}

\begin{theorem}\label{teo:HS} \emph{(Hierarchical spline approximants)}.
  Assume that $a \in S^m$, and that it is properly supported. Denote
  by $\tilde{a}$ the expansion of $a$ in hierarchical splines for
  $\xi$ (introduced in Section \ref{sec:HS}), and in a Fourier series
  for $x$, properly truncated to $M$ terms. Call $\tilde{A}$ and $A$
  the corresponding pseudodifferential operators on $L^2([0,1]^2)$
  defined by (\ref{eq:psido-sum}). Introduce $P_N$ the orthogonal
  projector onto frequencies obeying
  \[
  \max( |\xi_1|, |\xi_2|) \leq N.
  \]
  Then there exists a choice of $M$ obeying
  \[
  M \leq C \cdot \eps^{-2/(p+1)} \cdot \log N,
  \]
  where $p$ is the order of the spline interpolant, and for some $C > 0$, such that
  \[
  \| ( \tilde{A} - A)  P_N \|_{H^m([0,1]^2) \to L^2([0,1]^2)} \leq \epsilon.
  \]
\end{theorem}

The important point of these theorems is that $M$ is either constant
in $N$ (Theorem \ref{teo:Cheb}), or grows like $\log N$ (Theorem
\ref{teo:HS}), where $N$ is the bandlimit of the functions to which
the operator is applied.


\subsection{Symbol operations}

At the level of kernels, composition of operators is a simple
matrix-matrix multiplication. This property is lost when considering
symbols, but composition remains simple enough that the gains in
dealing with small interaction matrices $a_{\lambda,\mu}$ as in
(\ref{eq:sep0}) are far from being offset.

The \emph{twisted product} of two symbols $a$ and $b$, is the symbol
of their composition. It is defined as $(a \sharp b)(x,D) = a(x,D)
b(x,D)$ and obeys
\[
a \sharp b (x,\xi) = \int \int e^{-2 \pi i (x-y) \cdot (\xi - \eta)} a(x,\eta) b(y,\xi) \, dy d\eta.
\]
This formula holds for $\xi, \eta \in \R^d$, but in the case when
frequency space is discrete, the integral in $\eta$ is to be replaced by a sum.
In Section \ref{sec:dsc-op} we explain how to evaluate this
formula very efficiently using the symbol expansions discussed
earlier.

Textbooks on pseudodifferential calculus also describe asymptotic
expansions of $a \sharp b$ where negative powers of $|\xi|$ are
matched at infinity \cite{Hor, Fol, Sog}, but, as alluded to previously,
we are not interested in making simplifications of this kind.

Composition can be regarded as a building block for performing many
other operations using iterative methods. Functions of operators can
be computed by substituting the twisted product for the matrix-matrix
product in any algorithm that computes the corresponding function of a
matrix. For instance,
\begin{itemize}
\item The inverse of a positive-definite operator can be obtained via
  a Neumann iteration, or via a Schulz iteration;
\item There exist many choices of iterations for computing the square
  root and the inverse square root of a matrix \cite{Hig}, such as the
  Schulz-Higham iteration;
\item The exponential of a matrix can be obtained by the
  scaling-and-squaring method; etc.
\end{itemize}
These examples are discussed in detail in Section \ref{sec:dsc-op}. 

Two other operations that resemble composition from the algorithmic
viewpoint, are 1) transposition, and 2) the Moyal transform for
passing to the Weyl symbol.  They are also discussed below.

Lastly, this work would be incomplete without a routine for applying a
pseudodifferential operator to a function, from the knowledge of its
symbol. The type of separated expansion considered in equation
(\ref{eq:sep0}) suggests a very simple algorithm for this task,
detailed in Section \ref{sec:dsc-op}. (This part is not original;
it was already considered in previous work by Emmanuel Cand\`{e}s and
the authors in \cite{FastFIO}, where the more general case of Fourier
integral operators was considered.)

\subsection{Applications}

Applications of discrete symbol calculus abound in the numerical
solutions of linear partial differential equations (PDE) with variable coefficients. We outline several examples in this section and their numerical results are given in Section
\ref{sec:app}.

In all of these applications, our solution takes two steps. First, we use discrete symbol calculus to construct the symbol of the
operator which solves the PDE problem. Since the
data has not been queried yet (i.e., the right hand side, the initial conditions, or the
boundary conditions), the computational cost of this step is
mostly independent of the size of the data. Once the operator is ready in its
symbol form, we apply the operator to the data in the second step.

The two regimes in which this approach could be preferred is when
either 1) the complexity of the medium is low compared to the
complexity of the data, or 2) the problem needs to be solved several
times and benefits from being ``preconditioned'' in some way.

A first, toy application of discrete symbol calculus is to the numerical
solution of the simplest elliptic PDE,
\begin{equation}
  Au := (I-\div(\alpha(x)\grad) u = f
  \label{eqn:ellp}
\end{equation}
with $\alpha(x) > 0$, and periodic boundary conditions on a square. If
$\alpha(x)$ is a constant function, the solution requires only two
Fourier transforms, since the operator is diagonalized by the Fourier
basis. For variable $\alpha(x)$, discrete symbol calculus can be seen
as a natural generalization of this fragile Fourier diagonalization
property: we construct the symbol of $A^{-1}$ directly, and once the
symbol of $A^{-1}$ is ready, applying it to the function $f$ requires
only a small number of Fourier transforms.


The second application of discrete symbol calculus is related to
the Helmholtz equation
\begin{equation}
  L u := \left( -\Delta - \frac{\omega^2}{c^2(x)} \right) u = f(x)
  \label{eq:helm}
\end{equation}
where the sound speed $c(x)$ is a smooth function in $x$, in a
periodized square. The numerical solution of this problem is difficult
since the operator $L$ is not positive definite so that efficient
techniques such as multigrid cannot be used directly for this problem.
A standard iterative algorithm, such as MINRES or BIGGSTAB, can easily
take tens of thousands of iterations to converge.  One way to obtain
faster convergence is to solve a preconditioned system
\begin{equation}
  M^{-1} L u = M^{-1} f
  \label{eq:helm-p}
\end{equation}
with 
\[
M := -\Delta + \frac{\omega^2}{c^2(x)}
\quad\mbox{or}\quad
M := -\Delta + (1+i) \frac{\omega^2}{c^2(x)}
\]
Now at each iteration of the preconditioned system,
we need to invert a linear system for the preconditioner $M$.
Multigrid is typically used for this \cite{Erl}, but discrete symbol
calculus offers a way to directly precompute the symbol of $M^{-1}$.
Once it is ready, applying $M^{-1}$ to a function at each iteration is
reduced to a small number of Fourier transforms---three or four when
$c(x)$ is very smooth---which we anticipate to be very competitive vs.
a multigrid method.

Another important application of the discrete symbol calculus is to
``polarizing'' the initial condition of a linear hyperbolic system.
Let us consider the following variable coefficient wave equation on
the periodic domain $ x \in [0,1]^2$,
\begin{equation}
  \begin{cases}
    u_{tt} - \div (\alpha(x) \grad u) = 0\\
    u(0,x) = u_0(x) \\
    u_t(0,x) = u_1(x)
  \end{cases}
  \label{eq:wave}
\end{equation}
with the extra condition $\int u_1(x) d x = 0$. The operator $L:= -
\div (\alpha(x) \grad)$ is symmetric positive definite, and let us define
$P$ to be its square root $L^{1/2}$. We can then use $P$ to factorize the
wave equation as
\[
(\pd_t + i P) (\pd_t - i P) u = 0.
\]
As a result, the solution $u(t,x)$ can be represented as
\[
u(t,x) = e^{i t P} u_+(x) + e^{-i t P} u_-(x)
\]
where the polarized components $u_+(x)$ and $u_-(x)$ of the initial
condition are given by
\[
u_+ = \frac{u_0 + (iP)^{-1} u_1}{2}
\quad\mbox{and}\quad
u_- = \frac{u_0 - (iP)^{-1} u_1}{2}.
\]
To compute $u_+$ and $u_-$, we first use discrete symbol calculus to
construct the symbol of $P^{-1}$. Once the symbol of $P^{-1}$ is
ready, the computation of $u_+$ and $u_-$ requires only applying
$P^{-1}$ to the initial condition. Applying $e^{itP}$ is a more
difficult problem that we do not address in this paper.

Finally, discrete symbol calculus has a natural application to the
problem of depth extrapolation, or migration, of seismic data. In the
Helmholtz equation
\[
\Delta_\bot + \frac{\pd^2 u}{\pd z^2} + \frac{\omega^2}{c^2(x,z)} u = F(x,z,k),
\]
we can separate the Laplacian as $\Delta = \Delta_{\bot} + \frac{\pd^2}{\pd z^2}$, and factor the equation as
\begin{equation}\label{eq:ssr1}
\left( \frac{\pd}{\pd z} - B(z) \right) v = F(x,z,k) - \frac{\pd B}{\pd z}(z)u, \qquad \left( \frac{\pd}{\pd z} + B(z) \right) u = v
\end{equation}
where $B = \sqrt{- \Delta_\bot - \omega^2/c^2(x,z)}$ is called the
one-way wave propagator, or single square root (SSR) propagator. We
may then focus on the equation for $v$, called the SSR equation, and
solve it for decreasing $z$ from $z=0$. The term $\frac{\pd B}{\pd
  z}(z)u$ above is sometimes neglected, as we do in the sequel, on the
basis that it introduces no new singularities.

The symbol of $B^2$ is not elliptic; its zero level set presents a
well-known issue with this type of formulation. In Section
\ref{sec:app}, we introduce an adequate ``directional'' cutoff
strategy to remove the singularities that would otherwise appear,
hence neglect turning rays and evanescent waves, and then use discrete
symbol calculus to compute a well-behaved operator square root. We
then show how to solve the SSR equation approximately using an
operator exponential of $B$, also realized via discrete symbol
calculus. Unlike traditional methods of seismic imaging (discussed in
Section \ref{sec:related} below), the \emph{only} simplification we
make here is the directional cutoff just mentioned.

\subsection{Harmonic analysis of symbols}

It is instructive to compare the symbol expansions of this paper with
another type of expansion thought to be efficient for smooth
differential and integral operators, namely wavelets.

Consider $x \in [0,1]$ for simplicity. The standard matrix of an
operator $A$ in a basis of wavelets $\psi_{j,k}(x) = 2^{j/2} \psi(2^j
x-n)$ of $L^2([0,1])$ is simply $\< \psi_{j,k}, A \psi_{j',k'} \>$.
Such wavelet matrices were first considered by Meyer in \cite{Mey},
and later by Beylkin, Coifman, and Rokhlin in \cite{BCR}, for the
purpose of obtaining sparse expansions of singular integral operators
in the Calder\'{o}n-Zygmund class. Their result is that either $O(N)$
or $O(N \log N)$ elements suffice to represent a $N$-by-$N$ matrix
accurately, in the $\ell_2$ sense, in a wavelet basis. This result is
not necessarily true in other bases such as Fourier series or local
cosines, and became the basis for much activity in some numerical
analysis circles in the 1990s.

In contrast, the expansions proposed in this paper assume a class of
operators with symbols in the $S^m$ class defined (\ref{eq:type10}),
but achieve accurate compression with $O(1)$ or $O(\log N)$ elements,
way sublinear in $N$. This stark difference is illustrated in Figure
\ref{fig:sec14}.

\begin{figure}[htb]
  \begin{center}
    \begin{tabular}{cc}
      \includegraphics[height=2.5in]{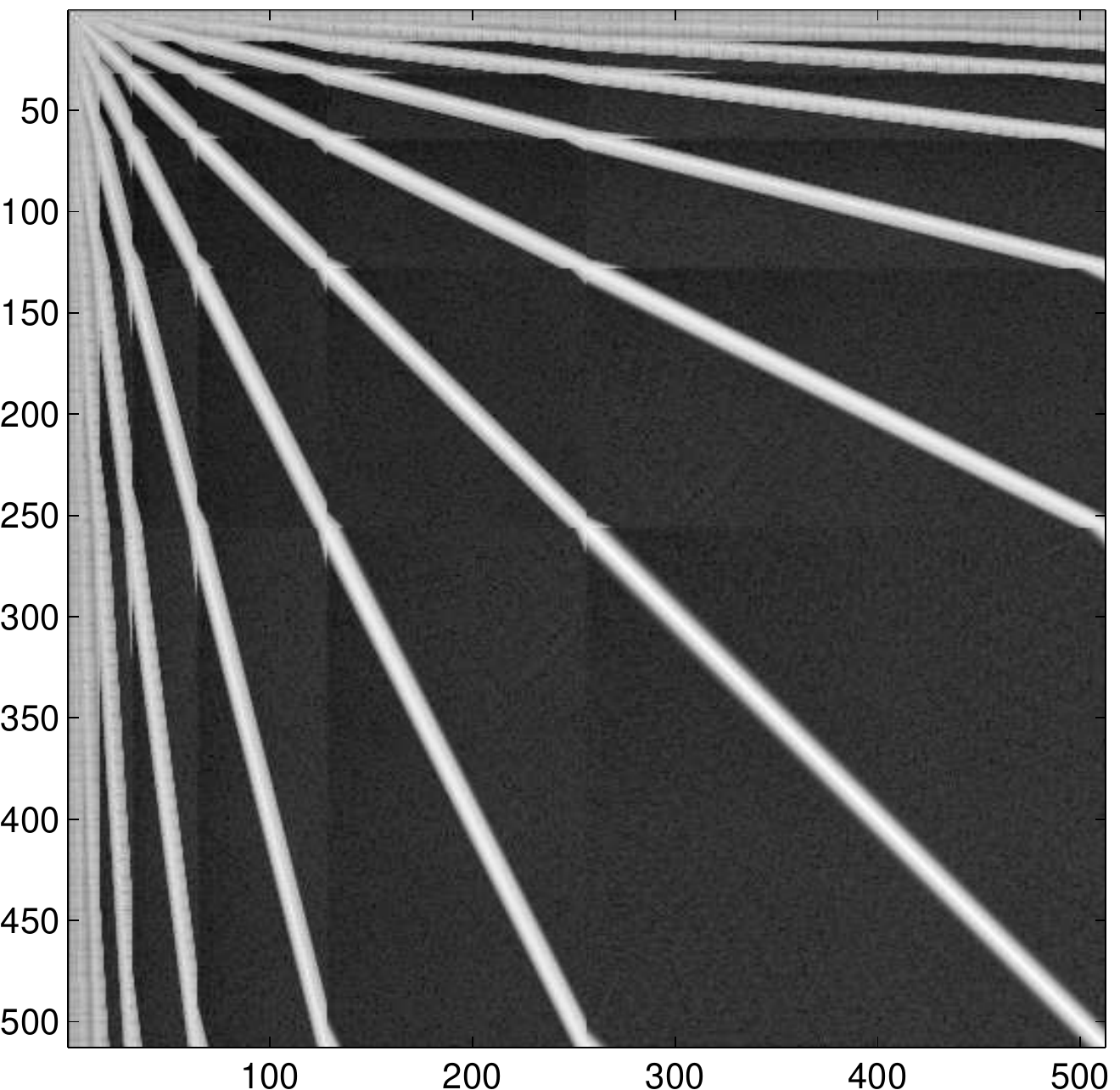} &  \includegraphics[height=0.38in]{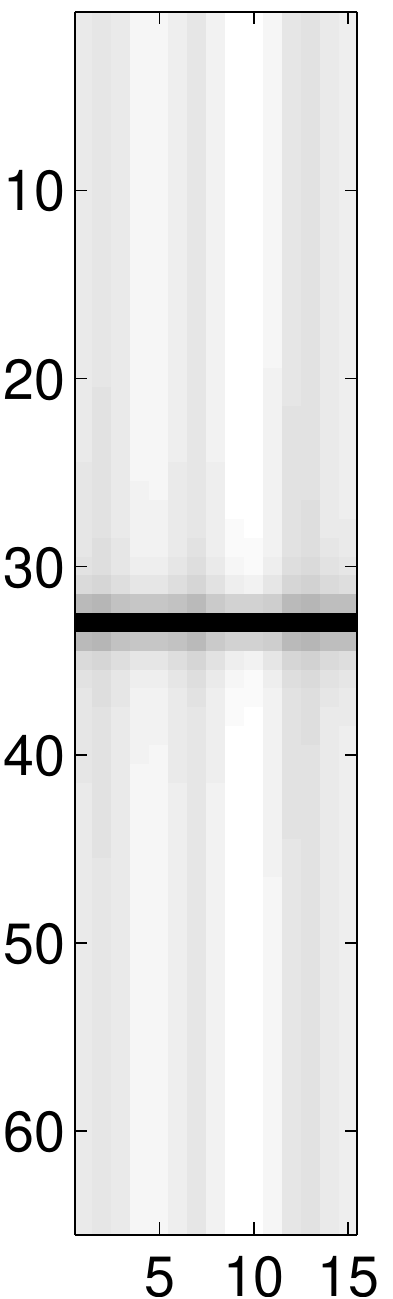}\\
      (a) & (b)
    \end{tabular}
  \end{center}
  \caption{Left: the standard $512$-by-$512$ wavelet matrix of the
    differential operator considered in Figure \ref{fig:sec11}, truncated
    to elements greater than $10^{-5}$ (white). Right: the
    $65$-by-$15$ interaction matrix of DSC, for the same operator and
    a comparable accuracy, using a hierarchical splines expansion in
    $\xi$. The scale is the same for both pictures. Notice that the
    DSC matrix can be further compressed by a singular value
    decomposition, and in this example has numerical rank equal to
    $3$, for a singular value cutoff at $10^{-5}$. For values of $N$
    greater than $512$, the wavelet matrix would increase in size in a
    manner directly proportional to $N$, while the DSC matrix would
    grow in size like $\log N$.}
  \label{fig:sec14}
\end{figure}

Tasks such as inversion and computing the square root are realized in
$O(\log^2 N)$ operations, still way sublinear in $N$. It is only when
the operator needs to be applied to functions defined on $N$ points,
as a ``post-computation'', that the complexity becomes $C \cdot N \log
N$. This constant $C$ is proportional to the numerical rank of the
symbol, and reflects the difficulty of storing it accurately, not the
difficulty of computing it. In practice, we have found that typical
values of $C$ are still much smaller than the constants that arise in
wavelet analysis, which are often plagued by a curse of dimensionality
\cite{Demanet-thesis}.

Wavelet matrices can sometimes be reduced in size to a mere $O(1)$
too, with controlled accuracy. To our knowledge this observation has
not been reported in the literature yet, and goes to show that some
care ought to be exercised before calling a method ``optimal''. The
particular smoothness properties of symbols that we leverage for their
expansion is also hidden in the wavelet matrix, as \emph{additional
  smoothness along the shifted diagonals.} The following result is
elementary and we give it without proof.

\begin{theorem}\label{teo:wavelet} Let $A \in \Psi^0$ as defined by
  (\ref{eq:type10}), for $x \in \R$ and $\xi \in \R$. Let $\psi_{j,k}$
  be an orthonormal wavelet basis of $L^2(\R)$ of class $C^\infty$,
  and with an infinite number of vanishing moments. Then for each $j$,
  and each $\Delta k = k - k'$, there exists a function $f_{j, \Delta
    k} \in C^\infty(\R)$ with smoothness constants independent of $j$,
  such that
\[
\< \psi_{j,k}, A \psi_{j,k'} \> = f_{j,\Delta k} (2^{-j} k).
\]
\end{theorem}

We would like to mention that similar ideas of smoothness along the
diagonal have appeared in the context of seismic imaging, for the
diagonal fitting of the so-called normal operator in a curvelet frame
\cite{HerrmannStolk, FDCT}. In addition, the construction of
second-generation bandlets for image processing is based on a similar
phenomenon of smoothness along edges for the unitary recombination of
MRA wavelet coefficients \cite{Peyre}. We believe that this last
``alpertization'' step could be of great interest in numerical
analysis.

Theorem \ref{teo:wavelet} hinges on the assumption of symbols in
$S^m$, which is not met in the more general context of
Calder\'{o}n-Zygmund operators (CZO) considered by Meyer, Beylkin,
Coifman, and Rokhlin. The class of CZO has been likened to a
limited-smoothness equivalent to symbols of type $(1,1)$ and order 0,
i.e., symbols that obey
\[
| \pd_\xi^\alpha \pd_x^\beta a(x,\xi)| \leq C_{\alpha,\beta} \< \xi \>^{-|\alpha| + |\beta|}.
\]
Symbols of type $(1,0)$ and order 0 obeying (\ref{eq:type10}) are a
special case of this. Wavelet matrices of operators in the $(1,1)$
class are almost diagonal\footnote{Their standard wavelet matrix has
  at most $O(j)$ large elements per row and column at scale $j$---or
  frequency $O(2^j)$---after which the matrix elements decay
  sufficiently fast below a preset threshold. $L^2$ boundedness would
  follow if there were $O(1)$ large elements per row and column, but
  $O(j)$ does not suffice for that, which testifies to the fact that
  operators of type $(1,1)$ are not in general $L^2$ bounded. The
  reason for this $O(j)$ number is that an operator with a $(1,1)$
  symbol does not preserve vanishing moments of a wavelet---not even
  approximately. Such operators may turn an oscillatory wavelet at any
  scale $j$ into a non-oscillating bump, which then requires wavelets
  at all the coarser scales for its expansion.}, but there is no
smoothness along the shifted diagonals as in Theorem
\ref{teo:wavelet}. So while the result in \cite{BCR} is sharp, namely
no much else than wavelet sparsity can be expected for CZO, we may
question whether the generality of the CZO class is truly needed for
applications to partial differential equations. The authors are
unaware of a PDE setup which requires the introduction of symbols in
the $(1,1)$ class that would not also belong to the $(1,0)$ class.

%
%







\subsection{Related work}\label{sec:related}

The idea of writing pseudodifferential symbols in separated form to
formulate various one-way approximations to the variable-coefficients
Helmholtz equation has long been a tradition in seismic imaging. This
almost invariably involves a high-frequency approximation of some
kind. Some influential work includes the phase screen method by Fisk
and McCartor \cite{Fisk}, and the generalized screen expansion of Le
Rousseau and de Hoop \cite{LeRdeH}. This last reference discusses fast
application of pseudodifferential operators in separated form using
the FFT, and it is likely not the only reference to make this simple
observation. A modern treatment of leading-order pseudodifferential
approximations to one-way wave equations is in \cite{Stolk2}.

Expansions of principal symbols $a_0(x, \xi/|\xi|)$ (homogeneous of
degree 0 is $\xi$) in spherical harmonics in $\xi$ is a useful tool in
the theory of pseudodifferential operators \cite{Tay}, and has also
been used for fast computations by Bao and Symes in \cite{Bao}. For computation of
pseudodifferential operators, see also the work by Lamoureux,
Margrave, and Gibson \cite{LMG}.

In the numerical analysis community, separation of operator kernels
and other high-dimensional functions is becoming an important topic.
Beylkin and Mohlenkamp proposed an alternated least-squares algorithm
for computing separated expansions of tensors in \cite{BeyMoh, BM}, propose to compute functions of operators in this representation, and apply these ideas to solving the multiparticle Schr\"{o}dinger equation
in \cite{BeyMohPer}, with Perez.

A different, competing approach to compressing operators is the
``partitioned separated'' method that consists in isolating
off-diagonal squares of the kernel $K(x,y)$, and approximating each of
them by a low-rank matrix. This also calls for an adapted notion of
calculus, e.g., for composing and inverting operators. The first
reference to this algorithmic framework is probably the partitioned
SVD method described in \cite{PSVD}. More recently, these ideas have
been extensively developed under the name H-matrix, for hierarchical
matrix; see \cite{borm1,hackbusch1} and http://www.hlib.org.

Separation ideas, with an adapted notion of operator calculus, have
also been suggested for solving the wave equation; two examples are
\cite{BeySan} and \cite{watuwe}.

Exact operator square-roots---up to numerical errors---have in some
contexts already been considered in the literature. See \cite{Fis}
for an example of Helmholtz operator with a quadratic profile, and
\cite{Lin} for a spectral approach that leverages sparsity, also
for the Helmholtz operator.

\section{Discrete Symbol Calculus: Representations}
\label{sec:dsc-rep}

The two central questions of discrete symbol calculus are:
\begin{itemize}
\item Given an operator $A$, how to represent its symbol $a(x,\xi)$ efficiently?
\item How to perform the basic operations of the pseudodifferential
  symbol calculus based on this representation? These operations
  include sum, product, adjoint, inversion, square root, inverse
  square root, and, in some cases, the exponential.
\end{itemize}

These two questions are mostly disjoint; we answer the first question
in this section, and the second question in Section \ref{sec:dsc-op}.



Let us write expansions of the form (\ref{eq:sep0}). Since
$e_\lambda(x) = e^{2 \pi i x \cdot \lambda}$ with $x \in [0,1]^2$, we
denote the $x$-Fourier series coefficients of $a(x,\xi)$ as
\[
\hat{a}_\lambda(\xi) = \int_{[0,1]^2} e^{- 2 \pi i x \cdot \lambda} a(x,\xi) \, dx, \qquad \lambda \in \Z^2.
\]
We find it convenient to write $h_{a,\lambda}(\xi) =\hat{a}_\lambda(\xi) \la\xi\ra^{-d_a}$, hence
\begin{equation}\label{eq:sep}
  a(x,\xi) = \sum_{\lambda} e_\lambda(x) h_{a,\lambda}(\xi) \la\xi\ra^{d_a}.
\end{equation}

In the case when $a(\cdot,\xi)$ is bandlimited with band $B_x$, i.e.,
$\hat{a}_\lambda(\xi)$ is supported inside the square $(-B_x,B_x)^2$
in the $\lambda$-frequency domain, then the integral can be computed
exactly by a uniform quadrature on the points $x_p = p/(2B_x)$, with
$0 \leq p_1,p_2 < 2B_x$. This grid is called $X$ in the sequel.

The problem is now reduced to finding adequate expansions
$\tilde{h}_{a,\lambda}$ for $h_{a,\lambda}$, either valid in the whole
plane $\xi \in \R^2$, or in a large square $\xi \in [-N,N]^2$.


\subsection{Rational Chebyshev interpolant}\label{sec:Cheb}

For symbols in the class (\ref{eq:classical}), each function
$h_{a,\lambda}(\xi) = \hat{a}_\lambda(\xi) \la\xi\ra^{-d_a}$ is smooth
in angle $\mbox{arg } \xi$, and polyhomogeneous in radius $|\xi|$.
This means that $h_{a,\lambda}$ is for $|\xi|$ large a polynomial of
$1/|\xi|$ along each radial line through the origin, and is otherwise
smooth (except possibly near the origin).

One idea for efficiently expanding such functions is to map the half
line $|\xi| \in [0,\infty)$ to the interval $[-1,1]$ by a rational
function, and expand the result in Chebyshev polynomials. Put $\xi =
(\theta, r)$, and $\mu = (m,n)$. Let
\[
g_\mu(\xi) =  e^{i m \theta} TL_{n}(r),
\]
where $TL_n$ are the rational Chebyshev functions
\cite{boyd-2001-cfsm}, defined from the Chebyshev polynomials of the
first kind $T_n$ as
\[
TL_n(r) = T_n(A^{-1}_L(r)),
\]
by means of the algebraic map
\[
s \mapsto r = A_L(s) = L \frac{1+s}{1-s}, \qquad r \mapsto s = A^{-1}_L(r) = \frac{r-L}{r+L}.
\]
The parameter $L$ is typically on the order of 1. The proposed
expansion then takes the form
\[
h_{a,\lambda}(\xi) = \sum_\mu a_{\lambda,\mu}  g_\mu(\xi),
\]
or $\tilde{h}_{a,\lambda}(\xi)$ if the sum is truncated, where
\begin{equation*}
a_{\lambda,\mu} = \frac{1}{2 \pi} \int_{-1}^1 \int_0^{2 \pi} h_{a,\lambda}((\theta,A_L(s))) e^{-i m \theta} T_n(s)
\, \frac{d\theta ds}{\sqrt{1-s^2}}.
\end{equation*}

For properly bandlimited functions, such integrals can be evaluated
exactly using the right quadrature points: uniform in $\theta \in [0,2
\pi]$, and Gauss points in $s$. The corresponding points in $r$
are the image of the Gauss points under the algebraic map. The
resulting grid in the $\xi$ plane can be described as follows. Let $q
= (q_\theta, q_r)$ be a couple of integers such that $0 \leq q_\theta
< N_\theta$ and $0 \leq q_r < N_r$; we have in polar coordinates
\[
\xi_{q} = \left( 2 \pi \frac{q_\theta}{N_\theta} , - \cos \left( \frac{2(A_L(q_r) - 1)}{2N_r} \right) \right).
\]
We call this grid $\{ \xi_q \} = \Omega$. Passing from the values $h_{a,\lambda}(\xi_q)$ to $a_{\lambda,\mu}$
and vice-versa can be done using the fast Fourier transform. Of course, $\tilde{h}_{a,\lambda}(\xi)$ is nothing but an interpolant of $h_{a,\lambda}(\xi)$ at the points $\xi_q$.

\bigskip

In the remainder of this section, we present the proof of Theorem
\ref{teo:Cheb}, which contains the convergence rates of the truncated
sums over $\lambda$ and $\mu$. The argument hinges on the following
$L^2$ boundedness result, which is a simple modification of standard
results in $\R^d$, see \cite{Ste}. It is not necessary to restrict
$d=2$ for this lemma.

\begin{lemma}\label{teo:L2bdd}
Let $a(x,\xi) \in C^{d'}([0,1]^d, \ell_\infty(\Z^d))$, where $d' = d+1$ if $d$ is odd, or $d+2$ if $d$ is even. Then the operator $A$ defined by (\ref{eq:psido-sum}) extends to a bounded operator  on $L^2([0,1]^d)$, with
\[
\| A \|_{L^2} \leq C \cdot \| (1 + (- \Delta_x)^{d'/2}) a(x,\xi) \|_{L^\infty([0,1]^d, \ell_\infty(\Z^d))}.
\]
\end{lemma}

The proof of this lemma is in the Appendix.

\begin{proof}[Proof of Theorem \ref{teo:Cheb}.]


  Consider $s = A^{-1}_L(r) \in [-1,1)$ where $A_L$ and its inverse
  were defined above. Expanding $a(x,(\theta, r))$ in rational
  Chebyshev functions $TL_n(r)$ is equivalent to expanding $f(s)
  \equiv a(x,(\theta,A_L(s)))$ in Chebyshev polynomials $T_n(s)$.
  Obviously,
  \[
  f \circ A_L^{-1} \in C^\infty([0,\infty)) \qquad \Leftrightarrow \qquad f \in C^\infty([-1,1)).
  \]
  
  It is furthermore assumed that $a(x,\xi)$ is in the classical class
  with tempered growth of the polyhomogeneous components; this condition
  implies that the smoothness constants of $f(s) = a(x,(\theta,A_L(s)))$
  are uniform as $s \to 1$, i.e., for all $n \geq 0$,
  \[
  \exists \; C_n \; : \quad |f^{(n)}(s)| \leq C_n, \quad s \in [-1,1],
  \]
  or simply, $f \in C^\infty([-1,1])$. In order to see why that is the
  case, consider a cutoff function $\chi(r)$ equal to 1 for $r \geq
  2$, zero for $0 \leq r \leq 1$, and $C^\infty$ increasing in
  between. Traditionally, the meaning of $(\ref{eq:classical})$ is
  that there exists a sequence $\eps_j > 0$, defining cutoffs $\chi(r
  \eps_j)$ such that
  \[
  a(x,(r,\theta)) - \sum_{j \geq 0} a_j(x,\theta) r^{-j} \chi ( r \eps_j) \in S^{-m}_{\cl},
  \]
  for all $m \geq 0$. A remainder in $S^{-\infty}_{\cl} \equiv
  \bigcup_{m \geq 0} S^{-m}_{\cl}$ is called smoothing. As long as the
  choice of cutoffs ensures convergence, the determination of
  $a(x,\xi)$ modulo $S^{-\infty}$ does not depend on this choice.
  (Indeed, if there existed an order-$k$ discrepancy between the sums
  with $\chi(r \eps_j)$ or $\chi(r \delta_j)$, with $k$ finite, it
  would need to come from some of the terms $a_j r^{-j} ( \chi(r
  \eps_j) - \chi(r \delta_j) )$ for $j \leq k$. But each of these
  terms is of order $- \infty$.)

  Because of condition (\ref{eq:tempered}), it is easy to check that
  the particular choice $\eps_j = 1/(2R)$ suffices for convergence of
  the sum over $j$ to a symbol in $S^0$. As mentioned above, changing
  the $\eps_j$ only affects the smoothing remainder, so we may focus
  on $\eps_j = 1/(2R)$.

After changing variables, we get
\[
f(s) = a(x,(\theta,A_L(s))) = \sum_{j \geq 0} a_j(x,\theta) L^{-j} \left( \frac{1-s}{1+s} \right)^j  \chi \left( \frac{A_L(s)}{2R} \right) + r(s),
\]
where the smoothing remainder $r(s)$ obeys
\[
|r^{(n)}(s)| \leq C_{n,M} (1-s)^M, \qquad \forall M \geq 0,
\]
hence, in particular when $M = 0$, has uniform smoothness constants as
$s \to 1$. It suffices therefore to show that the sum over $j \geq 0$
can be rewritten as a Taylor expansion for $f(s) - r(s)$, convergent
in some neighborhood of $s = 1$.

Let $z = 1-s$. Without loss of generality, assume that $R \geq 2L$,
otherwise increase $R$ to $2L$. The cutoff factor $ \chi \left(
  \frac{A_L(1-z)}{2R} \right)$ equals $1$ as long as $0 \leq z \leq
\frac{L}{4R}$. In that range,
\[
f(1-z) - r(1-z) = \sum_{j \geq 0} a_j(x,\theta) L^{-j} \frac{z^j}{(2-z)^j}.
\]
By making use of the binomial expansion
\[
\frac{z^j}{(2-z)^j} = \sum_{m \geq 0} \left( \frac{z}{2} \right)^{j+m} \begin{pmatrix} j+m-1 \\ j-1 \end{pmatrix}, \qquad \mbox{if } j \geq 1,
\]
and the new index $k = j+ m$, we obtain the Taylor expansion about $z = 0$:
\[
f(1-z) - r(1-z) = a_0(x,\theta) + \sum_{k \geq 0} \left( \frac{z}{2} \right)^{k} \sum_{1 \leq j \leq k} \frac{a_j(x,\theta)}{L^j}  \begin{pmatrix} k-1 \\ j-1 \end{pmatrix}.
\]
To check convergence, notice that $\begin{pmatrix} k-1 \\ j-1
\end{pmatrix} \leq \sum_{n = 0}^{k-1} \begin{pmatrix} k-1 \\ n
\end{pmatrix} = 2^{k-1}$, combine this with (\ref{eq:tempered}), and
obtain
\begin{align*}
2^{-k} \sum_{1 \leq j \leq k} \frac{a_j(x,\theta)}{L^j}  \begin{pmatrix} k-1 \\ j-1 \end{pmatrix} &\leq \frac{Q_{00}}{2} \sum_{1 \leq j \leq k} \left( \frac{R}{L} \right)^{j} \\
&\leq \frac{Q_{00}}{2} \frac{1}{1 - L/R} \left( \frac{R}{L} \right)^k
\end{align*}
We assumed earlier that $z \in [0, L/(4R)]$: this condition manifestly
suffices for convergence of the sum over $k$.  This shows that $f \in
C^{\infty}([-1,1])$; the very same reasoning with $Q_{\alpha \beta}$
in place of $Q_{00}$ also shows that any derivative $\pd_x^\alpha
\pd_\theta^\beta f(s) \in C^\infty([-1,1])$.

The Chebyshev expansion of $f(s)$ is the Fourier-cosine series of
$f(\cos \phi)$, with $\phi \in [0,\pi]$; the previous reasoning shows
that $f(\cos \phi) \in C^\infty([0,\infty])$. The same is true for any
$(x,\theta)$ derivatives of $f(\cos \phi)$.

Hence $a(x,(A_L(\cos \phi),\theta))$ is a $C^\infty$ function,
periodic in all its variables. The proposed expansion scheme is
simply:
\begin{itemize}
\item A Fourier series in $x \in [0,1]^2$;
\item A Fourier series in $\theta \in [0,2 \pi]$;
\item A Fourier-cosine series in $\phi \in [0,\pi]$.
\end{itemize}
An approximant with at most $M$ terms can then be defined by keeping
$\lfloor M^{1/4} \rfloor$ Fourier coefficients per direction.  It is
well-known that Fourier and Fourier-cosine series of a $C^\infty$,
periodic function converge super-algebraically in the $L^\infty$ norm,
and that the same is true for any derivative of the function as well.
Therefore if $a_M$ is this $M$-term approximant, we have
\[
\sup_{x,\theta,\phi} | \pd_x^\beta (a-\tilde{a})(x,(A_L(\cos \phi),\theta))| \leq C_{\beta, M} \cdot M^{-\infty}, \qquad \forall \mbox{ multi-index } \beta
\]

We now invoke Lemma \ref{teo:L2bdd} with $a - a_M$ in place of $a$,
choose $M = O(\eps^{-1/\infty})$ with the right constants, and
conclude.

\end{proof}

It is interesting to observe what goes wrong when condition
(\ref{eq:tempered}) is not satisfied. For instance, if the growth of
the $a_j$ is fast enough in (\ref{eq:classical}), then it may be
possible to choose the cutoffs $\chi(\eps_j |\xi|)$ such that the sum
over $j$ replicates a fractional negative power of $|\xi|$, like
$|\xi|^{-1/2}$, and in such a way that the resulting symbol is still
in the class defined by (\ref{eq:type10}). A symbol with this kind of
decay at infinity would not be mapped onto a $C^\infty$ function of
$s$ inside $[-1,1]$ by the algebraic change of variables $A_L$, and
the Chebyshev expansion in $s$ would not converge spectrally. This
kind of pathology is generally avoided in the literature on
pseudodifferential operators by assuming that the order of the
compound symbol $a(x,\xi)$ is the same as that of the principal
symbol, i.e., the leading-order contribution $a_0(x, \mbox{arg }
\xi)$.

Finally, note that the obvious generalization of the complex
exponentials in arg $\xi$ to higher-dimensional settings would be
spherical harmonics, as advocated in \cite{Bao}. The radial expansion
scheme should probably remain unchanged, though.

\subsection{Hierarchical spline interpolant}\label{sec:HS}

An alternative representation is to use a hierarchical spline
construction in the $\xi$ plane. We define
$\tilde{h}_{a,\lambda}(\xi)$ to be an interpolant of
$h_{a,\lambda}(\xi) = \hat{a}_\lambda(\xi) \la\xi\ra^{-d_a}$ as
follows. We only define the interpolant in the square $\xi \in
[-N,N]^2$ for some large $N$. Pick a number $B_{\xi}$---independent of
$N$---that plays the role of coarse-scale bandwidth; in practice it is
taken comparable to $B_x$.

\begin{itemize}
\item Define $D_0 = (-B_\xi,B_\xi)^2$. For each $\xi \in D_0$,
  $\tilde{h}_{a,\lambda}(\xi) := h_{a,\lambda}(\xi)$.
\item For each $j=1,2,\cdots,L=\log_3(N/B_\xi)$, define $D_j = (-3^j B_\xi,
  3^j B_\xi)^2 - D_{j-1}$. We further partition $D_j$ into eight blocks:
  \[
  D_j = \bigcup_{i=1}^{8} D_{j,i},
  \]
  where each block $D_{j,i}$ is of size $2\cdot 3^{j-1}B_\xi \times
  2\cdot 3^{j-1}B_\xi$. Within each block $D_{j,i}$, we sample
  $h_{a,\lambda}(\xi)$ with a Cartesian grid $G_{j,i}$ of a fixed
  size.  The restriction of $\tilde{h}_{a,\lambda}(\xi)$ in $D_{j,i}$
  is defined to be the spline interpolant of $h_{a,\lambda}(\xi)$ on
  the grid $G_{j,i}$.
\end{itemize}

\begin{figure}[h!]
  \begin{center}
    \begin{tabular}{cc}
      \includegraphics[height=2.5in]{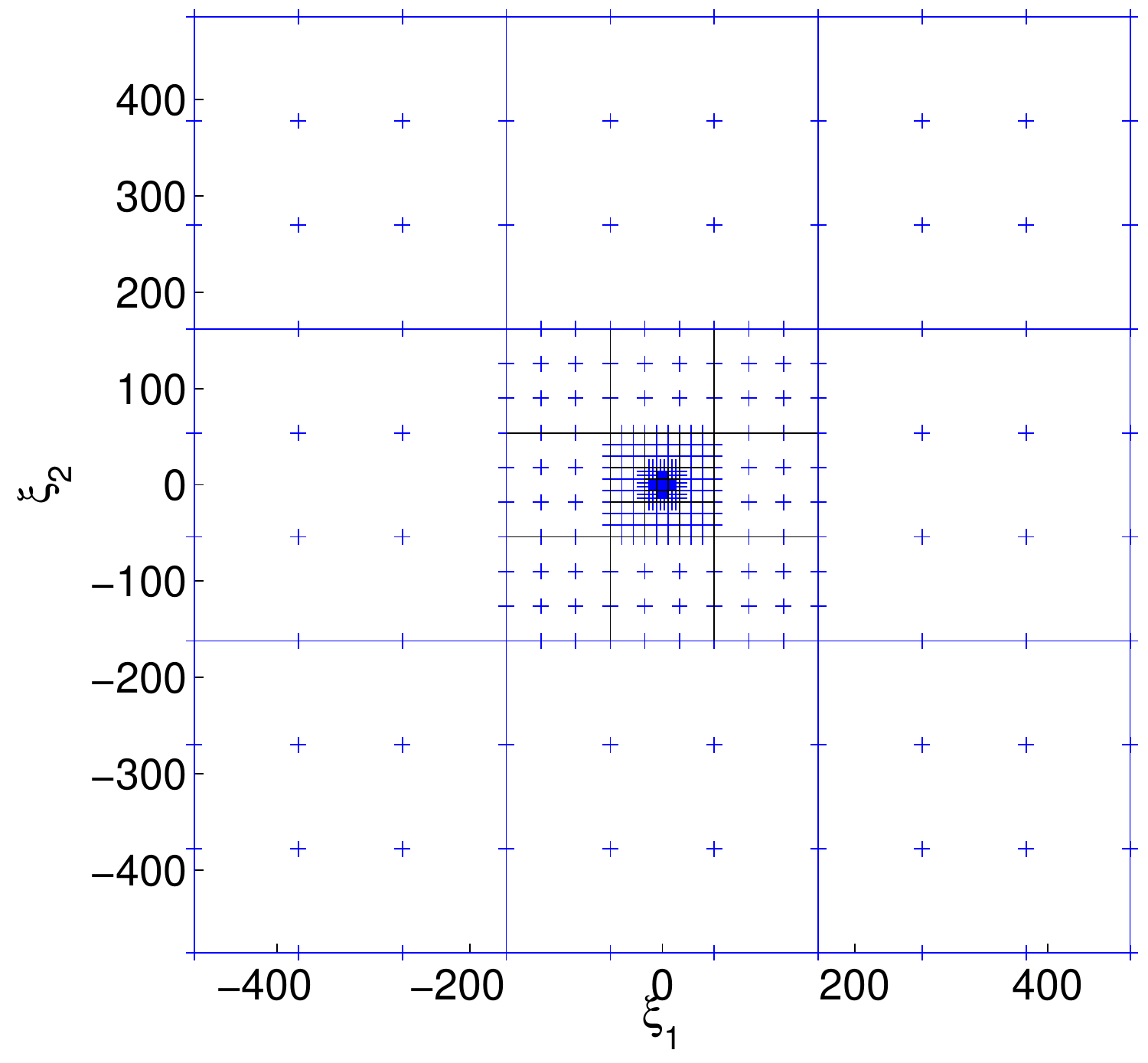} &
      \includegraphics[height=2.5in]{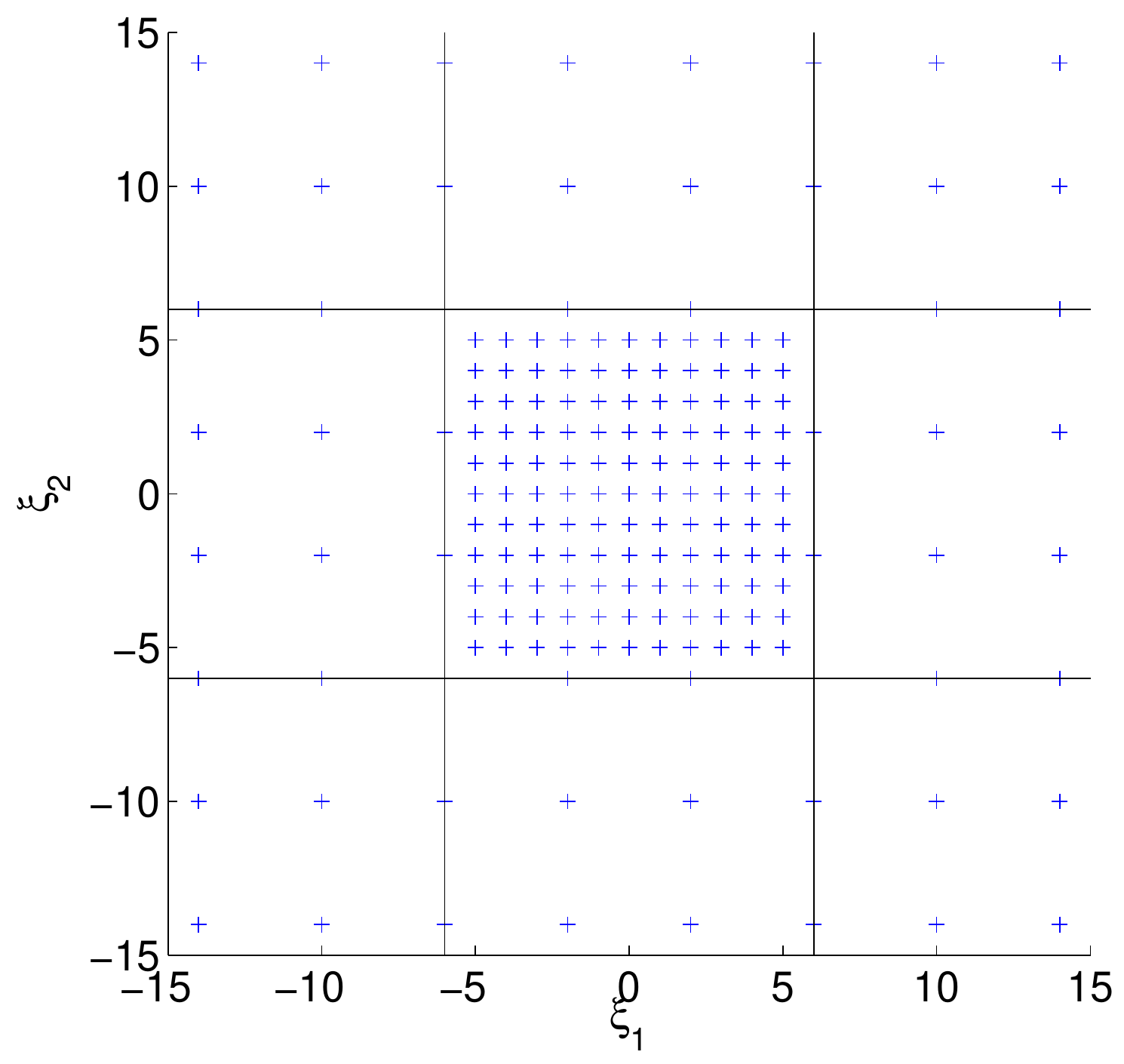}\\
      (a) & (b)
    \end{tabular}
  \end{center}
  \caption{ Hierarchical spline construction.  Here $B_\xi=6$, $L=4$, and
    $N=486$. The grid $G_{j,i}$ is of size $4 \times 4$.  The grid
    points are shown with ``+'' sign.  (a) The whole grid.  (b) The
    center of the grid.  }
  \label{fig:hsp}
\end{figure}

We emphasize that the number of samples used in each grid $G_{j,i}$ is
fixed independent of the level $j$. The reason for this is that the function
$h_{a,\lambda}(\xi)$ gains smoothness as $\xi$ grows to
infinity. In practice, we set $G_{j,i}$ to be a $4 \times 4$ or $5
\times 5$ Cartesian grid and use cubic spline interpolation.

Let us summarize the construction of the representation $a(x,\xi)
\approx \sum_\lambda e_\lambda(x) \tilde{h}_{a,\lambda}(\xi)
\la\xi\ra^{d_a}$. As before fix a parameter $B_x$ that governs the
bandwidth in $x$, and define
\[
X = \left\{ \left(\frac{p_1}{2B_x}, \frac{p_2}{2B_x}\right), 0 \le p_1,p_2< 2B_x \right\}
\quad\mbox{and}\quad
\Omega = D_0 \bigcup \left( \bigcup_{j,i} G_{j,i} \right).
\]
The construction of the expansion of $a(x,\xi)$ takes the following
steps
\begin{itemize}
\item Sample $a(x,\xi)$ for all pairs of $(x,\xi)$ with $x \in X$ and
  $\xi \in \Omega$.
\item For a fixed $\xi \in \Omega$, use the fast Fourier transform to
  compute $\hat{a}_\lambda (\xi) $ for all $\lambda\in (-B_x,B_x)^2$.
\item For each $\lambda$, construct the interpolant $\tilde{h}_{a,\lambda}(\xi)$ from the
  values of $h_{a,\lambda}(\xi) = \hat{a}_\lambda(\xi) \la\xi\ra^{-d_a}$ at $\xi \in \Omega$.
\end{itemize}

Let us study the complexity of this construction procedure. The number
of samples in $X$ is bounded by $4 B_x^2$, considered a constant with
respect to $N$. As we use a constant number of samples for each level
$j=1,2,\cdots,L=\log_3(N/B_\xi)$, the number of samples in $\Omega$ is
of order $O(\log N)$. Therefore, the total number of samples is still
of order $O(\log N)$. Similarly, since the construction of a fixed
size spline interpolant requires only a fixed number of steps, the
construction of the interpolants $\{ \tilde{h}_{a,\lambda}(\xi) \}$
takes only $O(\log N)$ steps as well. Finally, we would like to remark
that, due to the locality of the spline, the evaluation of
$\tilde{h}_{a,\lambda}(\xi)$ for any fixed $\lambda$ and $\xi$
requires only a constant number of steps.

\bigskip

We now expand on the convergence properties of the spline interpolant.

\begin{proof}[Proof of Theorem \ref{teo:HS}]

If the number of control points per square $D_{j,i}$ is $K^2$ instead
of 16 or 25 as we advocated above, the spline interpolant becomes
arbitrarily accurate. The spacing between two control points at level
$j$ is $O(3^{j}/K)$. With $p$ be the order of the spline scheme---we
took $p=3$ earlier---it is standard polynomial interpolation theory
that
\[
\sup_{\xi \in D_{j,i}} |\tilde{h}_{a,\lambda}(\xi) - h_{a,\lambda}(\xi)| \leq C_{a,\lambda,p} \cdot \left( \frac{3^j}{K} \right)^{p+1} \cdot \sup_{|\alpha| = p+1} \| \pd_\xi^\alpha h_{a,\lambda} \|_{L^\infty(D_{j,i})}.
\]
The symbol estimate (\ref{eq:type10}) guarantees that the last factor
is bounded by $C \cdot \sup_{\xi \in D_{j,i}} \la\xi\ra^{-p-1}$. Each
square $D_{j,i}$, for fixed $j$, is at a distance $O(3^j)$ from the
origin, hence $\sup_{\xi \in D_{j,i}} \la\xi\ra^{-p-1} =
O(3^{-j(p+1)})$. This results in
\[
\sup_{\xi \in D_{j,i}} |\tilde{h}_{a,\lambda}(\xi) - h_{a,\lambda}(\xi)| \leq C_{a,\lambda,p} \cdot K^{-p-1}.
\]
This estimate is uniform over $D_{j,i}$, hence also over $\xi \in
[-N,N]^2$. As argued earlier, it is achieved by using $O(K^2 \log N)$
spline control points. If we factor in the error of expanding the
symbol in the $x$ variable using $4 B^2$ spatial points, for a total
of $M = O(B^2 K^2 \log N)$ points, we have the compound estimate
\[
\sup_{x \in [0,1]^2} \sup_{\xi \in [-N,N]^2} | a(x,\xi) - \tilde{a}(x,\xi) | \leq C \cdot ( B^{-\infty} + K^{-p-1} ).
\]
The same estimate holds for the partial derivatives of $a - \tilde{a}$ in $x$. 

Functions to which the operator defined by $\tilde{a}(x,\xi)$ is
applied need to be bandlimited to $[-N,N]^2$, i.e., $\hat{f}(\xi) = 0$
for $\xi \notin [-N,N]^2$, or better yet $f = P_N f$. For those
functions, the symbol $\tilde{a}$ can be extended by $a$ outside of
$[-N,N]^2$, Lemma \ref{teo:L2bdd} can be applied to the difference $A
- \tilde{A}$, and we obtain
\[
\| (A - \tilde{A}) f \|_{L^2} \leq C \cdot ( B^{-\infty} + K^{-p-1} ) \cdot \| f \|_{L^2}
\]
The leading factors of $\| f \|_{L^2}$ in the right-hand side can be
made less than $\eps$ if we choose $B = O(\eps^{-1/\infty})$ and $K =
O(\eps^{-1/(p+1)})$, with adequate constants. The corresponding number
of points in $x$ and $\xi$ is therefore $M = O(\eps^{-2/(p+1)} \cdot
\log N)$.

\end{proof}

\section{Discrete Symbol Calculus: Operations}
\label{sec:dsc-op}

Let $A$ and $B$ be two operators with symbols $a(x,\xi)$ and
$b(x,\xi)$. Suppose that we have already generated their expansions
\[
a(x,\xi) \approx \sum_\lambda e_\lambda(x) \tilde{h}_{a,\lambda}(\xi) \la\xi\ra^{d_a}
\quad\mbox{and}\quad
b(x,\xi) \approx \sum_\lambda e_\lambda(x) \tilde{h}_{b,j}(\xi) \la\xi\ra^{d_b},
\]
where $d_a$ and $d_b$ are the orders of $a(x,\xi)$ and $b(x,\xi)$,
respectively. It is understood that the sum over $\lambda$ is
truncated, that $h_{a,\lambda}(\xi)$ are approximated with
$\tilde{h}_{a,\lambda}(\xi)$ by either method described earlier, and
that we will not keep track of which particular method is used in the notations.

Let us now consider the basic operations of the calculus of discrete
symbols.

\paragraph{Scaling} $C = \alpha A$. For the symbols, we have
$
c(x,\xi) = \alpha a(x,\xi).
$
In terms of the Fourier coefficients,
\[
\hat{c}_\lambda(\xi) = \alpha \hat{a}_\lambda(\xi) \approx \alpha
\tilde{h}_{a,\lambda}(\xi) \la\xi\ra^{d_a}.
\]
Then $d_c = d_a$ and
\[
\tilde{h}_{c,\lambda}(\xi) := \tilde{h}_{a,\lambda}(\xi).
\]

\paragraph{Sum} $C = A + B$. For the symbols, we have
$
c(x,\xi) = a(x,\xi) + b(x,\xi).
$
In terms of the Fourier coefficients,
\[
\hat{c}_\lambda(\xi) = \hat{a}_\lambda(\xi) + \hat{b}_\lambda(\xi) \approx
\tilde{h}_{a,\lambda}(\xi) \la\xi\ra^{d_a} + \tilde{h}_{b,\lambda}(\xi) \la\xi\ra^{d_b}.
\]
Then $d_c = \max(d_a,d_b)$ and $\tilde{h}_{c,\lambda}(\xi)$ is the interpolant that
takes the values
\[
\left( \tilde{h}_{a,\lambda}(\xi) \la\xi\ra^{d_a} + \tilde{h}_{b,\lambda}(\xi) \la\xi\ra^{d_b} \right) \la\xi\ra^{-d_c}
\]
at $\xi \in \Omega$, where $\Omega$ is either the Gaussian points grid, or the hierarchical spline grid defined earlier.

\paragraph{Product} $C = A B$. For the symbols, we have
\[
c(x,\xi) = a(x,\xi) \sharp b(x,\xi) = \sum_{\eta} \int e^{-2\pi i (x-y)(\xi-\eta)} a(x,\eta) b(y,\xi) d y.
\]
In terms of the Fourier coefficients,
\[
\hat{c}_\lambda(\xi) = \sum_{k+l=\lambda} \hat{a}_k(\xi+l) \hat{b}_l(\xi) \approx 
\sum_{k+l=\lambda} \tilde{h}_{a,k}(\xi+l) \la\xi+l\ra^{d_a} \tilde{h}_{b,l}(\xi) \la\xi\ra^{d_b}.
\]
Then $d_c = d_a + d_b$ and $\tilde{h}_{c,\lambda}(\xi)$ is the interpolant that
takes the values
\[
\left( \sum_{k+l=\lambda} \tilde{h}_{a,k}(\xi+l) \la\xi+l\ra^{d_a} \tilde{h}_{b,l}(\xi) \la\xi\ra^{d_b} \right) \la\xi\ra^{-d_c}
\]
at $\xi \in \Omega$.

\paragraph{Transpose} $C = A^*$. For the symbols, it is straightforward to derive the formula
\[
c(x,\xi) = \sum_{\eta} \int e^{-2\pi i (x-y)(\xi-\eta)} \overline{a(y,\eta)} d y.
\]
In terms of the Fourier coefficients,
\[
\hat{c}_\lambda(\xi) = \overline{\hat{a}_{-\lambda}(\xi+\lambda)} \approx
\overline{\tilde{h}_{a,-\lambda}(\xi+\lambda)} \la\xi+\lambda\ra^{d_a}.
\]
Then $d_c = d_a$ and $\tilde{h}_{c,\lambda}(\xi)$ is the interpolant that takes the
values
\[
\left( \overline{\tilde{h}_{a,-\lambda}(\xi+\lambda)} \la\xi+\lambda\ra^{d_a} \right) \la\xi\ra^{-d_c}
\]
at $\xi \in\Omega$.

\paragraph{Inverse} $C = A^{-1}$ where $A$ is symmetric positive
definite. We first pick a constant $\alpha$ such that $\alpha
|a(x,\xi)| \ll 1$ for $\xi \in (-N, N)^2$. Since the order of $a(x,\xi)$
is $d_a$, $\alpha \approx O(1/N^{d_a})$. In the following iteration,
we first invert $\alpha A$ and then scale the result by $\alpha$ to
get $C$.
\begin{itemize}
\item $X_0 = I$.
\item For $k=0,1,2,\ldots$, repeat $X_{k+1} = 2 X_k - X_k (\alpha A)
  X_k$ until convergence.
\item Set $C = \alpha X_k$.
\end{itemize}

This iteration is called the Schulz iteration, and is quoted in
\cite{BeyMoh}. It can be seen as a modified Newton iteration for
finding the nontrivial zero of $f(X) = XAX-X$, where the gradient of
$f$ is approximated by the identity.


As this algorithm only utilizes the addition and the product of the
operators, all of the computation can be carried out via discrete
symbol calculus.  Since $\alpha \approx O(1/N^{d_a})$, the smallest
eigenvalue of $\alpha A$ can be as small as $O(1/N^{d_a})$ where the
constant depends on the smallest eigenvalue of $A$. For a given
accuracy $\eps$, it is not difficult to show that this algorithm
converges after $O(\log N + \log(1/\eps))$ iterations.


\paragraph{Square root and inverse square root} Put $C = A^{1/2}$ and $D
= A^{-1/2}$ where $A$ is symmetric positive definite. Here, we again
choose a constant $\alpha$ such that $\alpha |a(x,\xi)| \ll 1$ for $\xi
\in (-N, N)^2$. This also implies that $\alpha \approx O(1/N^{d_a})$.
In the following iteration, we first use the Schulz-Higham iteration
to compute the square root and the inverse square root of
$\alpha A$ and then scale them appropriately to obtain $C$ and $D$.
\begin{itemize}
\item $Y_0 = \alpha A$ and $Z_0 = I$.
\item For $k=0,1,2,\ldots$, repeat $Y_{k+1} = \frac{1}{2} Y_k (3I-Z_k
  Y_k)$ and $Z_{k+1} = \frac{1}{2} (3I- Z_k Y_k) Z_k$ until
  convergence.
\item Set $C = \alpha^{-1/2} Y_k$ and $D = \alpha^{1/2} Z_k$.
\end{itemize}
We refer to \cite{Hig} for a detailed discussion of this iteration.


In a similar way to the iteration used for computing the inverse, the Schulz-Higham
iteration is similar to the iteration for computing the inverse in that 
it uses only additions and products. Therefore, all of
the computation can be performed via discrete
symbol calculus.  A similar analysis show that, for any fixed accuracy
$\eps$, the number of iterations required by the Schulz-Higham iteration is
of order $O(\log N + \log(1/\eps))$ as well.

\paragraph{Exponential} $C = e^{\alpha A}$. In general, the exponential of an 
elliptic pseudodifferential operator is not necessarily a
pseudodifferential operator itself. However, if the data is restricted
to $\xi \in (-N,N)^2$ and $\alpha= O(1/N^{d_a})$, the exponential
operator behaves almost like a pseudodifferential operator in this range of frequencies\footnote{Note that another case in which the exponential remains pseudodifferential is when
the spectrum of $A$ is real and negative, regardless of the size of $\alpha$.}. In Section \ref{sec:sdm}, we will give an example where such an exponential operator plays an important role. 

We construct $C$ using
the following ``scaling-and-squaring" steps \cite{MolVan}:
\begin{itemize}
\item Pick $\delta$ sufficient small so that $\alpha/\delta = 2^K$ for an integer $K$.
\item Construct an approximation $Y_0$ for $e^{\delta A}$. One
  possible choice is the 4th order Taylor expansion: $Y_0 = I + \delta
  A + \frac{(\delta A)^2}{2!} + \frac{(\delta A)^3}{3!} +
  \frac{(\delta A)^4}{4!}$. Since $\delta$ is sufficient small, $Y_0$
  is quite accurate.
\item For $k=0,1,2,\ldots,K-1$, repeat $Y_{k+1} = Y_k Y_k$. 
\item Set $C = Y_K$.
\end{itemize}

This iteration for computing the exponential again uses only the
addition and product operations and, therefore, all the steps can be
carried out at the symbol level using discrete symbol calculus. The
number of steps $K$ is usually quite small, as the constant $\alpha$
itself is of order $O(1/N^{d_a})$.

\paragraph{Moyal transform} Pseudodifferential operators are sometimes defined by means of their \emph{Weyl symbol} $a_W$, as
\[
A f(x) = \sum_{\xi \in \Z^d} \int_{[0,1]^d]} a_W(\frac{1}{2}(x+y),\xi)
e^{2 \pi i (x-y) \xi} f(y) \, dy,
\]
when $\xi \in \Z^d$, otherwise if $\xi \in \R^d$, replace the sum over
$\xi$ by an integral. It is a more symmetric formulation that may be
preferred in some contexts. The other, usual formulation we have used
throughout this paper is called the Kohn-Nirenberg correspondence. The
relationship between the two methods of ``quantization'', i.e.,
passing from a symbol to an operator, is the so-called Moyal
transform. The book \cite{Fol} gives the recipe:
\[
a_W(x,\xi) = (M a)(x,\xi) = 2^n \sum_{\eta \in \Z^d} \int e^{4 \pi i (x-y) \cdot (\xi - \eta)} a(y,\eta) \, dy,
\]
and conversely
\[
a(x,\xi) = (M^{-1} a_W)(x,\xi) = 2^n \sum_{\eta \in \Z^d} \int e^{-4 \pi i (x-y) \cdot (\xi - \eta)} a(y,\eta) \, dy.
\]

These operations are algorithmically very similar to transposition. It
is interesting to notice that transposition is a mere conjugation in
the Weyl domain: $a^* = M^{-1} (\overline{M a})$. We also have the
curious property that
\[
\widehat{Ma}(p,q) = e^{-\pi i p q} \hat{a}(p,q)
\]
where the hat denotes Fourier transform in both variables.

\paragraph{Applying the operator} The last operation that we discuss
is how to apply the operator to a given input function. Suppose $u(x)$
is sampled on a grid $x = (p_1/N,p_2/N)$ with $0 \le p_1,p_2 < N$. Our
goal is to compute $(Au)(x)$ on the same grid. Using the definition of
the pseudo-differential symbol and the expansion of $a(x,\xi)$, we
have
\begin{eqnarray*}
  (Au)(x) &=& \sum_{\xi} e^{2\pi i x\xi} a(x,\xi) \hat{u}(\xi)\\
  &\approx& \sum_{\xi} e^{2\pi i x\xi} \sum_\lambda e_\lambda(x) \tilde{h}_{a,\lambda}(\xi) \la\xi\ra^{d_a} \hat{u}(\xi)\\
  &=& \sum_\lambda e_\lambda(x) \left( \sum_{\xi} e^{2\pi i x\xi} \left( \tilde{h}_{a,\lambda}(\xi) \la\xi\ra^{d_a} \hat{u}(\xi)\right) \right).
\end{eqnarray*}
Therefore, a straightforward yet efficient way to compute $Au$ is
\begin{itemize}
\item For each $\lambda \in (-B_x,B_x)^2 $, sample $\tilde{h}_{a,\lambda}(\xi)$ for $\xi \in [-N/2,N/2)^2$.
\item For each $\lambda \in (-B_x,B_x)^2 $, form the product $\tilde{h}_{a,\lambda}(\xi) \la\xi\ra^{d_a}
  \hat{u}(\xi)$ for $\xi \in [-N/2,N/2)^2$.
\item For each $\lambda \in (-B_x,B_x)^2 $, apply the fast Fourier transform to the result of
  the previous step.
\item For each $\lambda \in (-B_x,B_x)^2 $, multiply the result of the previous step with
  $e_\lambda(x)$. Finally, their sum gives $(Au)(x)$.
\end{itemize}

Let us estimate the complexity of this procedure. For each fixed $\lambda$,
the number of operations is dominated by the complexity of the fast
Fourier transform, which is $O(N\log N)$. Since there is only a
constant number of values for $\lambda \in (-B_x,B_x)^2 $, the overall
complexity is also $O(N\log N)$.

In many cases, we need to calculate $(Au)(x)$ for many different
functions $u(x)$. Though the above procedure is quite efficient, we
can further reduce the number of the Fourier transforms required.  The
idea is to exploit the possible redundancy between the functions
$\tilde{h}_{a,\lambda}(\xi)$ for different $\lambda$. We first use a rank-reduction
procedure, such as QR factorization or singular value decomposition
(SVD), to obtain a low-rank approximation
\[
\tilde{h}_{a,\lambda}(\xi) \approx \sum_{t=1}^T u_{\lambda t} v_t(\xi)
\]
where the number of terms $T$ is often much smaller than the possible
values of $j$. We can then write
\begin{eqnarray*}
  (Au)(x) &\approx& \sum_j e_\lambda(x) \sum_{\xi} e^{2\pi i x\xi} \sum_{t=1}^T u_{\lambda t} v_t(\xi) \la\xi\ra^{d_a} \hat{u}(\xi)\\
  &=& \sum_{t=1}^T \left(\sum_j e_\lambda(x) u_{\lambda t}\right)  \left( \sum_\xi e^{2\pi i x\xi} v_t(\xi) \la\xi\ra^{d_a} \hat{u}(\xi) \right).
\end{eqnarray*}
The new version of applying $(Au)(x)$ then takes two steps. In the
preprocessing step, we compute
\begin{itemize}
\item For each $\lambda \in (-B_x,B_x)^2 $, sample $\tilde{h}_{a,\lambda}(\xi)$ for $\xi \in [-N/2,N/2)^2$.
\item Construct the factorization $\tilde{h}_{a,\lambda}(\xi) \approx \sum_{t=1}^T u_{\lambda t} v_t(\xi)$.
\item For each $t$, compute the function $\sum_\lambda e_\lambda(x) u_{\lambda t}$.
\end{itemize}
In the evaluation step, one carries out the following steps for an input function $u(x)$
\begin{itemize}
\item For each $t$, compute $v_t(\xi) \la\xi\ra^{d_a} \hat{u}(\xi)$.
\item For each $t$, perform fast Fourier transform to the result of the previous step.
\item For each $t$, multiply the result with $\sum_\lambda e_\lambda(x) u_{\lambda t}$.
  Their sum gives $(Au)(x)$.
\end{itemize}

\section{Applications and Numerical Results}
\label{sec:app}

In this section, we provide several numerical examples to demonstrate
the effectiveness of the discrete symbol calculus. In these numerical
experiments, we use the hierarchical spline version of the discrete
symbol calculus. Our implementation is written in Matlab and all the
computational results are obtained on a desktop computer with 2.8GHz
CPU.

\subsection{Basic operations}

We first study the performance of the basic operations described in
Section \ref{sec:dsc-op}. In the following tests, we set $R = 6$,
$L=6$, and $N=R\times 3^L = 4374$. The number of samples in $\Omega$
is equal to 677. We consider the elliptic operator
\[
A u := (I - \div (\alpha(x) \grad)) u.
\]

\paragraph{Example 1.} The coefficient $\alpha(x)$ is a simple sinusoid
function given in Figure \ref{fig:op2d_1}. We use the discrete symbol
calculus to compute the operators $C = AA$, $C = A^{-1}$, and $ C=
A^{1/2}$. Table \ref{tbl:op2d_1} summarizes the running time, the
number of iterations, and the accuracy of these operations. We estimate
the accuracy using random noise as test functions. For a given test
function $f$, the errors are computed using the following quantities:
\begin{itemize}
\item For $C=A A$, we use $ \frac{\|Cf-A(Af)\|}{\|A(Af)\|}$.
\item For $C=A^{-1}$, we use $ \frac{\|A(Cf)-f\|}{\|f\|}$.
\item For $C=A^{1/2}$, we use $ \frac{\|C(Cf)-Af\|}{\|Af\|}$.
\end{itemize}

\begin{figure}[h!]
  \begin{center}
    \includegraphics[height=1.5in]{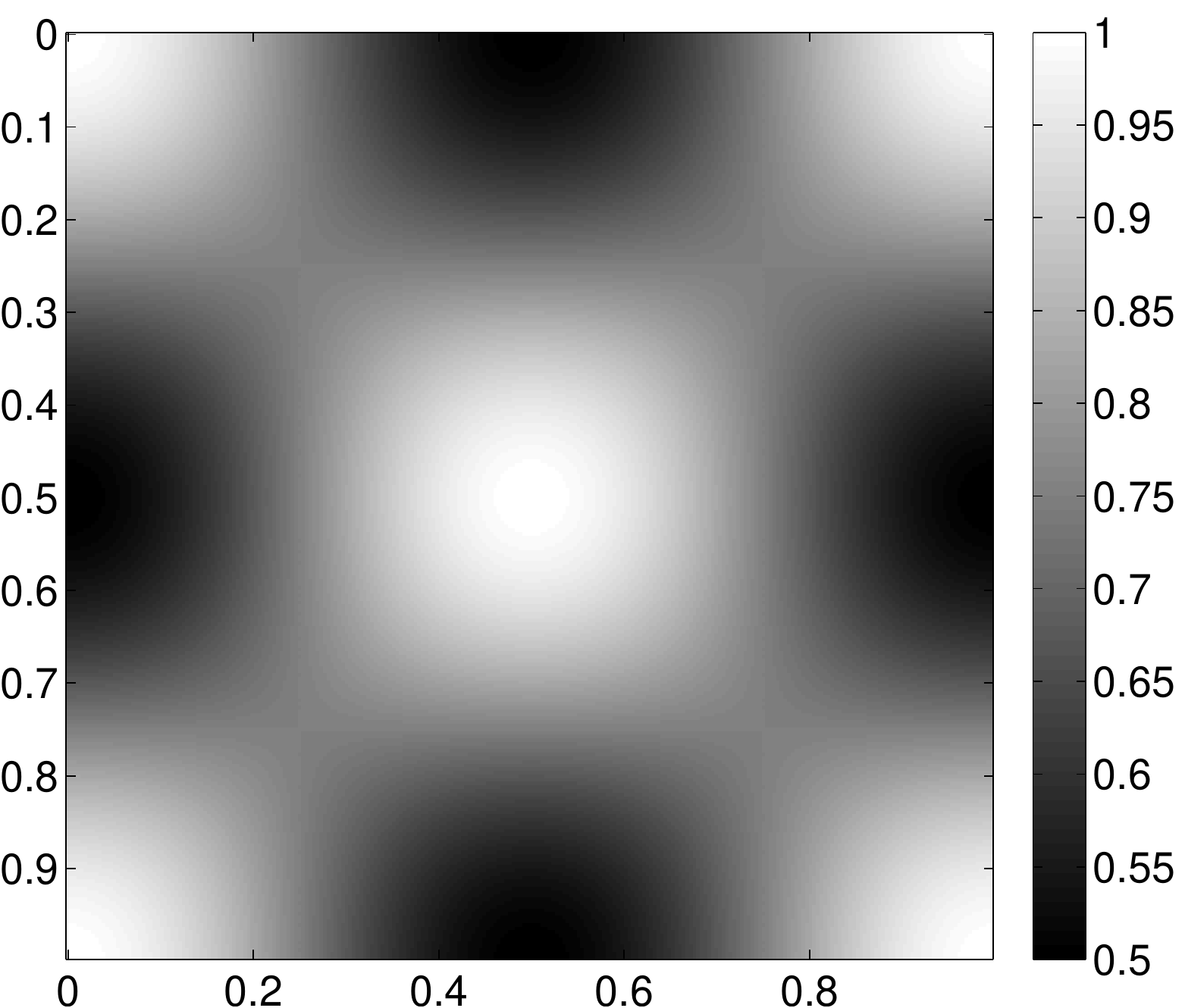}\\
  \end{center}
  \caption{The coefficient $\alpha(x)$ of Example 1.}
  \label{fig:op2d_1}
\end{figure}

\begin{table}[h!]
  \begin{center}
    \begin{tabular}{|c|ccc|}
      \hline
      & $\sharp$ iterations & Time & Accuracy\\
      \hline
      $C = AA$      &  - & 3.66e+00 & 1.92e-05\\
      $C = A^{-1}$  & 17 & 1.13e+02 & 2.34e-04\\
      $C = A^{1/2}$ & 27 & 4.96e+02 & 4.01e-05\\
      \hline
    \end{tabular}
  \end{center}
  \caption{The results of Example 1.}
  \label{tbl:op2d_1}
\end{table}

\paragraph{Example 2.} In this example, we set $\alpha(x)$ to be a random
bandlimited function (see Figure \ref{fig:op2d_2}). We report the
running time, the number of iterations, and the accuracy for each
operation in Table \ref{tbl:op2d_2}.

\begin{figure}[h!]
  \begin{center}
    \includegraphics[height=1.5in]{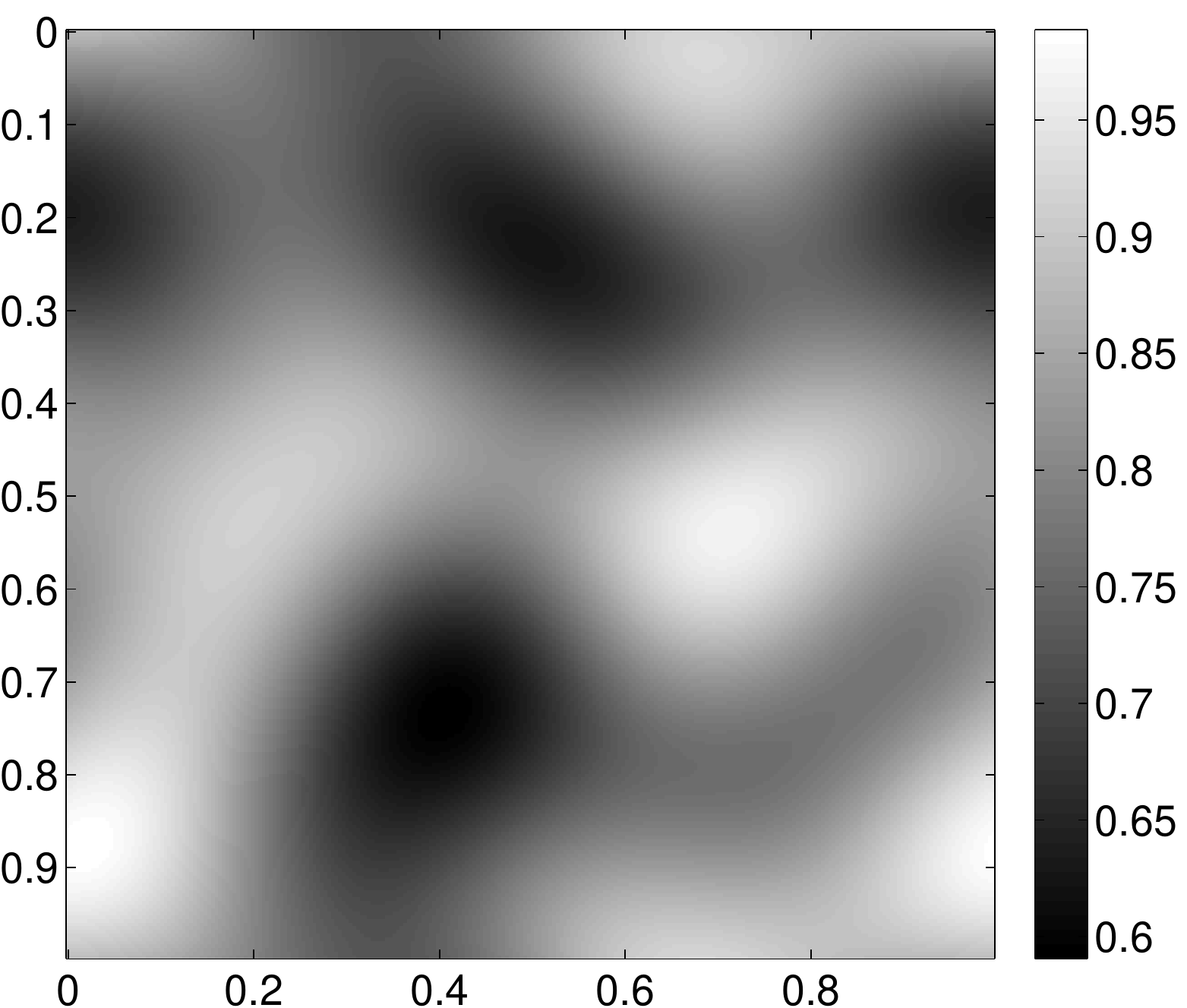}\\
  \end{center}
  \caption{The coefficient $\alpha(x)$ of Example 2.}
  \label{fig:op2d_2}
\end{figure}

\begin{table}[h!]
  \begin{center}
    \begin{tabular}{|c|ccc|}
      \hline
      & $\sharp$ iterations & Time & Accuracy\\
      \hline
      $C = AA$      &  - & 3.66e+00 & 1.73e-05\\
      $C = A^{-1}$  & 16 & 1.05e+02 & 6.54e-04\\
      $C = A^{1/2}$ & 27 & 4.96e+02 & 8.26e-05\\
      \hline
    \end{tabular}
  \end{center}
  \caption{The results of Example 2.}
  \label{tbl:op2d_2}
\end{table}

Tables \ref{tbl:op2d_1} and \ref{tbl:op2d_2} show that the number of
iterations for the inverse and square root operator remain almost
independent of the function $a(x,\xi)$. Our algorithms produce good
accuracy with a small number of sampling points in both $x$ and $\xi$.
Although we have $N=4374$ in these examples, we can increase the value
of $N$ easily either by adding several extra levels in the hierarchical
spline construction or by adding a couple more radial quadrature
points in the rational Chebyshev polynomial construction. For both of
these approaches, the running time and iteration count increase only
slightly: it is possible to show that they depend on $N$ in a logarithmic way.

\subsection{Preconditioner}

As we mentioned in the Introduction, an important application
of the discrete symbol calculus is to precondition the inhomogeneous
Helmholtz equation:
\[
L u := \left( -\Delta - \frac{\omega^2}{c^2(x)} \right) u = f
\]
where the sound speed $c(x)$ is smooth and periodic in $x$. We
consider the solution of the preconditioned system
\[
M^{-1} L u = M^{-1} f
\]
with the so-called complex-shifted Laplace preconditioner \cite{Erl}, of which we consider two variants,
\[
M_1 := -\Delta + \frac{\omega^2}{c^2(x)}
\quad\mbox{and}\quad
M_2 := -\Delta + (1+i)\cdot\frac{\omega^2}{c^2(x)}.
\]

For each preconditioner $M_j$ with $j=1,2$, we use the discrete symbol
calculus to compute the symbol of $M_j^{-1}$. In our case, applying $M_j^{-1}$
requires at most four fast Fourier transforms. Furthermore, since
$M_j^{-1}$ only serves as a preconditioner, we do not need to be very
accurate about applying $M_j^{-1}$. This allows us to further reduce
the number of terms in the expansion of the symbol of $M_j^{-1}$.

\paragraph{Example 3.} The sound speed $c(x)$ of this example is given
in Figure \ref{fig:hm2d_1}. We perform the test on different
combination of $\omega$ and $N$ with $\omega/N$ fixed. For both the
unconditioned and conditioned systems, we use BICGSTAB and set the
relative error to be $10^{-3}$.  The numerical results are given in
Table \ref{tbl:hm2d_1}. For each test, we report the number of
iterations and, in parenthesis, the running time, both for the unconditioned system and the
preconditioned system with $M_1$ and $M_2$.

\begin{figure}[h!]
  \begin{center}
    \includegraphics[height=1.5in]{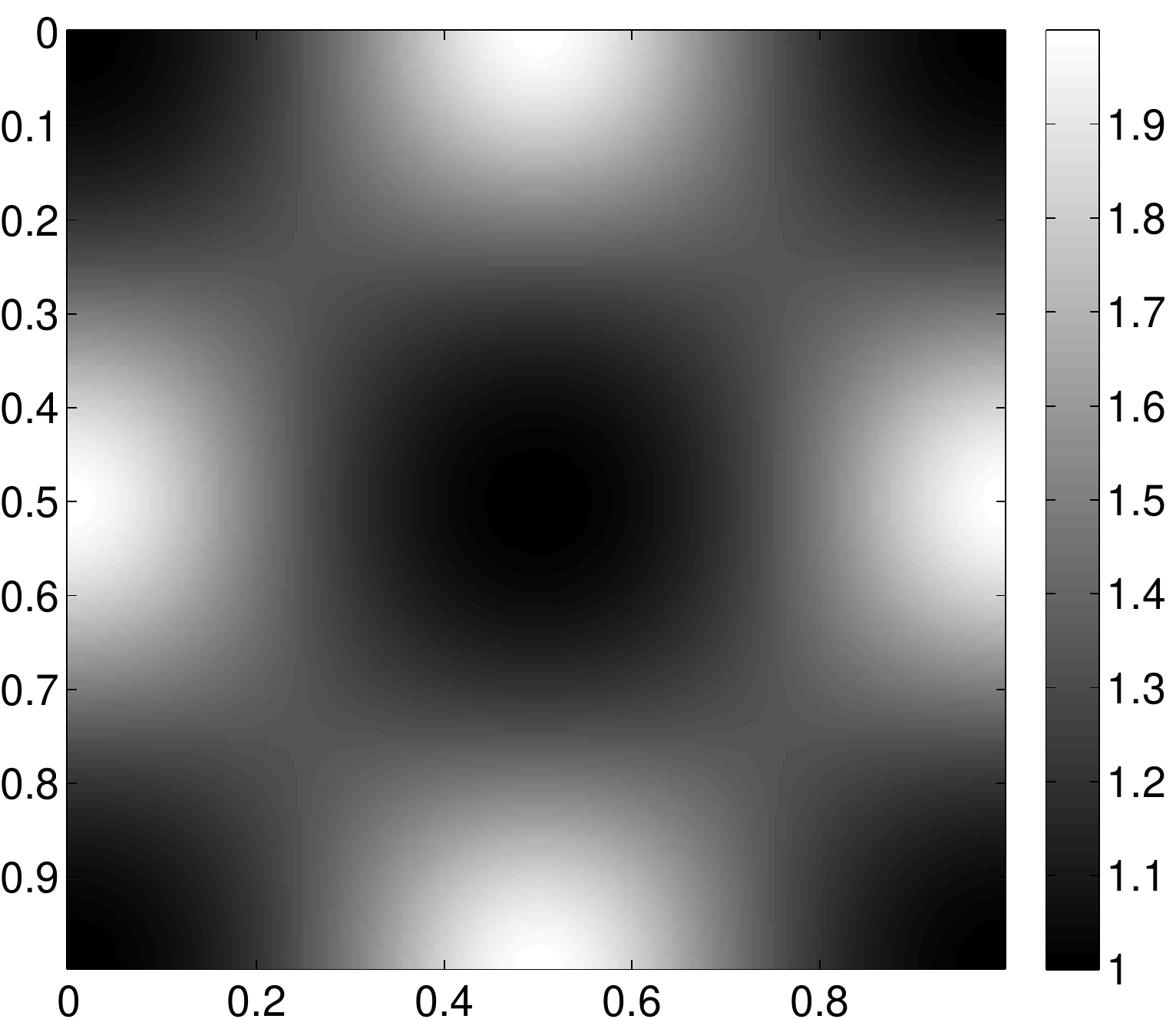}\\
  \end{center}
  \caption{The sound speed $c(x)$ of Example 3.}
  \label{fig:hm2d_1}
\end{figure}

\begin{table}[h!]
  \begin{center}
    \begin{tabular}{|c|ccc|}
      \hline
      $(w/2\pi,N)$  & Unconditioned & $M_1$ & $M_2$ \\
      \hline
      (4,64) 	& 2.24e+03 (8.40e+00)		&8.55e+01 (6.40e-01)	&5.70e+01 (5.10e-01)\\
      (8,128)	& 5.18e+03 (6.79e+01)		&1.50e+02 (4.16e+00)	&8.85e+01 (2.46e+00)\\
      (16,256) 	& 1.04e+04 (6.50e+02)		&4.98e+02 (6.79e+01)	&3.54e+02 (4.82e+01)\\
      (32,512)	&				&9.00e+02 (6.41e+02)	&3.06e+02 (2.20e+02)\\
      \hline
    \end{tabular}
  \end{center}
  \caption{
    The results of Example 3. For each test, we list the number of iterations and the 
    running time (in parenthesis).
  }
  \label{tbl:hm2d_1}
\end{table}

\paragraph{Example 4.} In this example, the sound speed $c(x)$ (shown
in Figure \ref{fig:hm2d_2}) is a Gaussian waveguide. We perform the
similar tests and the numerical results are summarized in Table
\ref{tbl:hm2d_2}.

\begin{figure}[h!]
  \begin{center}
    \includegraphics[height=1.5in]{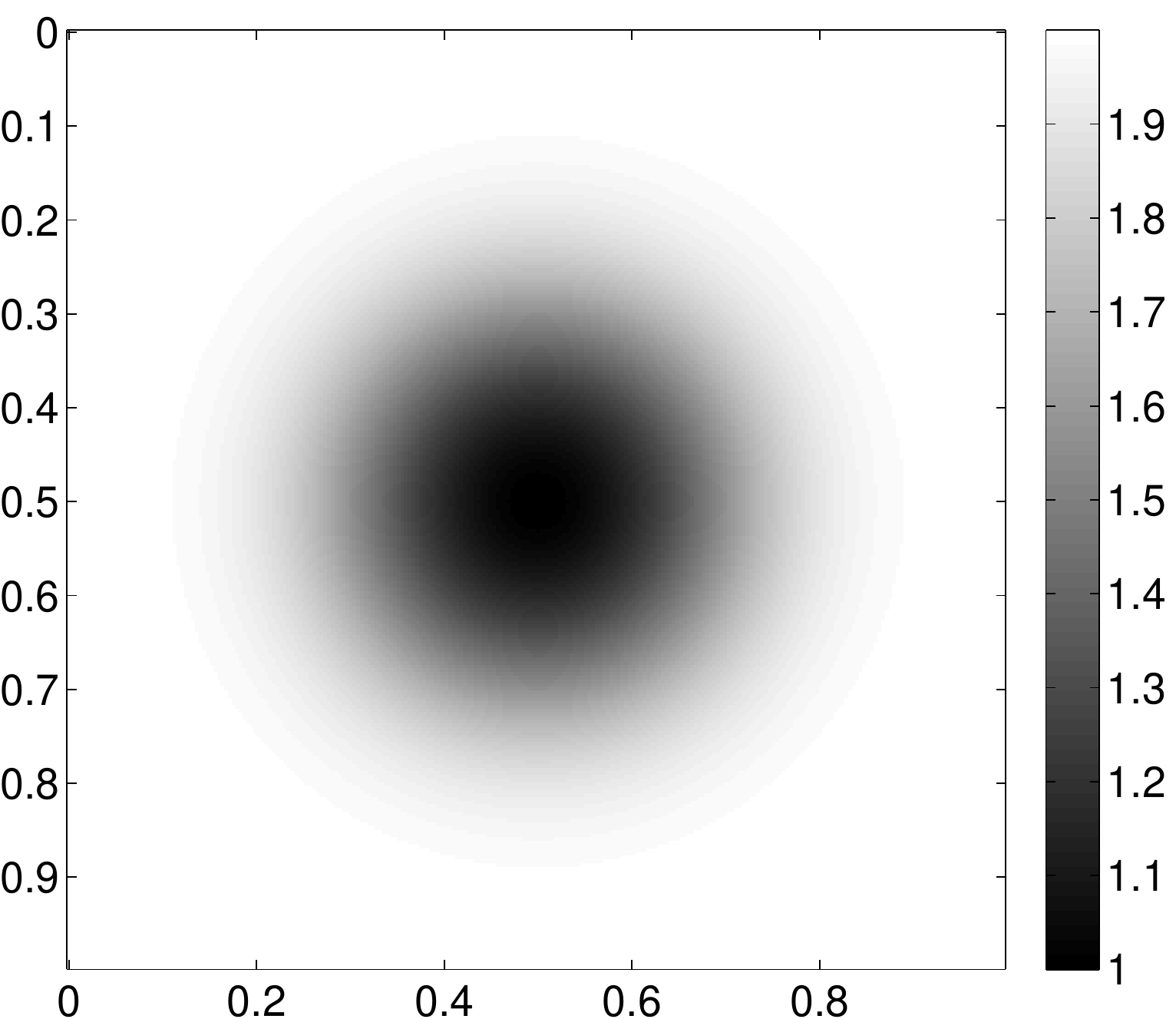}\\
  \end{center}
  \caption{The sound speed $c(x)$ of Example 4.}
  \label{fig:hm2d_2}
\end{figure}

\begin{table}[h!]
  \begin{center}
    \begin{tabular}{|c|ccc|}
      \hline
      $(w/2\pi,N)$  & Unconditioned & $M_1$ & $M_2$ \\
      \hline
      (4,64)	&3.46e+03 (1.30e+01)		&6.75e+01 (5.00e-01)	&4.25e+01 (3.20e-01)\\
      (8,128)	&1.06e+04 (1.39e+02)		&2.10e+02 (5.80e+00)	&1.16e+02 (3.19e+00)\\
      (16,256)	&3.51e+04 (1.93e+03)		&1.56e+03 (2.16e+02)	&6.81e+02 (9.56e+01)\\
      (32,512)	&		                &1.55e+03 (1.12e+03)	&6.46e+02 (4.63e+02)\\
      \hline
    \end{tabular}
  \end{center}
  \caption{
    The results of Example 4. For each test, we list the number of iterations and the 
    running time (in parenthesis).
  }    
  \label{tbl:hm2d_2}
\end{table}

In each of these two examples, we are able to use only 2 to 3 terms in
the expansion of the symbol of $M_j^{-1}$. The results in Tables
\ref{tbl:hm2d_1} and \ref{tbl:hm2d_2} show that the preconditioners
$M_1$ and $M_2$ reduce the number of iterations by a factor of 20 to
50, and the running time by a factor of 10 to 25.  We also observe
that the preconditioner $M_2$ outperforms $M_1$ by a factor of 2, in
line with observations in \cite{Erl}, where the complex constant
appearing in front of the $\omega^2/c^2(x)$ term in $M_1$ and $M_2$
was optimized.

Let us also note that we only consider the complex-shifted Laplace
preconditioner in isolation, without implementing any additional
deflation technique. Those seem to be very important in practice
\cite{ErlNab}.



\subsection{Polarization of wave operator}

Another application of the discrete symbol calculus is to ``polarize''
the initial condition of linear hyperbolic systems. We consider the
second order wave equation with variable coefficients
\[
  \begin{cases}
    u_{tt} - \div (\alpha(x) \grad u) = 0\\
    u(0,x) = u_0(x) \\
    u_t(0,x) = u_1(x)
  \end{cases}
\]
with the extra condition $\int u_1(x) d x = 0$. Since the operator
$L:= - \div (\alpha(x) \grad)$ is symmetric positive definite, its square
root $P := L^{1/2}$ is well defined. We can use $P$ to factorize the
equation into 
\[
(\pd_t + i P) (\pd_t - i P) u = 0.
\]
The solution $u(t,x)$ can be represented as
\[
u(t,x) = e^{i t P} u_+(x) + e^{-i t P} u_-(x)
\]
where the polarized components $u_+(x)$ and $u_-(x)$ of the initial
condition are given by
\[
u_+ = \frac{u_0 + (iP)^{-1} u_1}{2}
\quad\mbox{and}\quad
u_- = \frac{u_0 - (iP)^{-1} u_1}{2}.
\]
We first use the discrete symbol calculus to compute the operator
$P^{-1}$.  Once $P^{-1}$ is available, the computation of $u_+$ and
$u_-$ is straightforward.

\paragraph{Example 5. } The coefficient $\alpha(x)$ in this example is
shown in Figure \ref{fig:po2d_1} (a). The initial condition is set to
be a plane wave solution of the unit sound speed:
\[
u_0(x) = e^{2 \pi i k x}
\quad\mbox{and}\quad
u_1(x) = -2 \pi i |k| e^{2 \pi i k x},
\]
where $k$ is a fixed wave number. If $\alpha(x)$ were equal to $1$
everywhere, this initial condition itself would be polarized and the
component $u_+(x)$ would be zero. However, due to the inhomogeneity in
$\alpha(x)$, we expect both $u_+$ and $u_-$ to be non-trivial after the
polarization. The real part of $u_+(x)$ is plotted in Figure
\ref{fig:po2d_1} (b). We notice that the amplitude $u_+(x)$ scales
with the difference between the coefficient $\alpha(x)$ and $1$. This is
compatible with the asymptotic analysis of the operator $P$ for large
wave number. The figure of $u_-(x)$ is omitted as visually it is close
to $u_0(x)$.

\begin{figure}[h!]
  \begin{center}
    \begin{tabular}{cc}
      \includegraphics[height=1.5in]{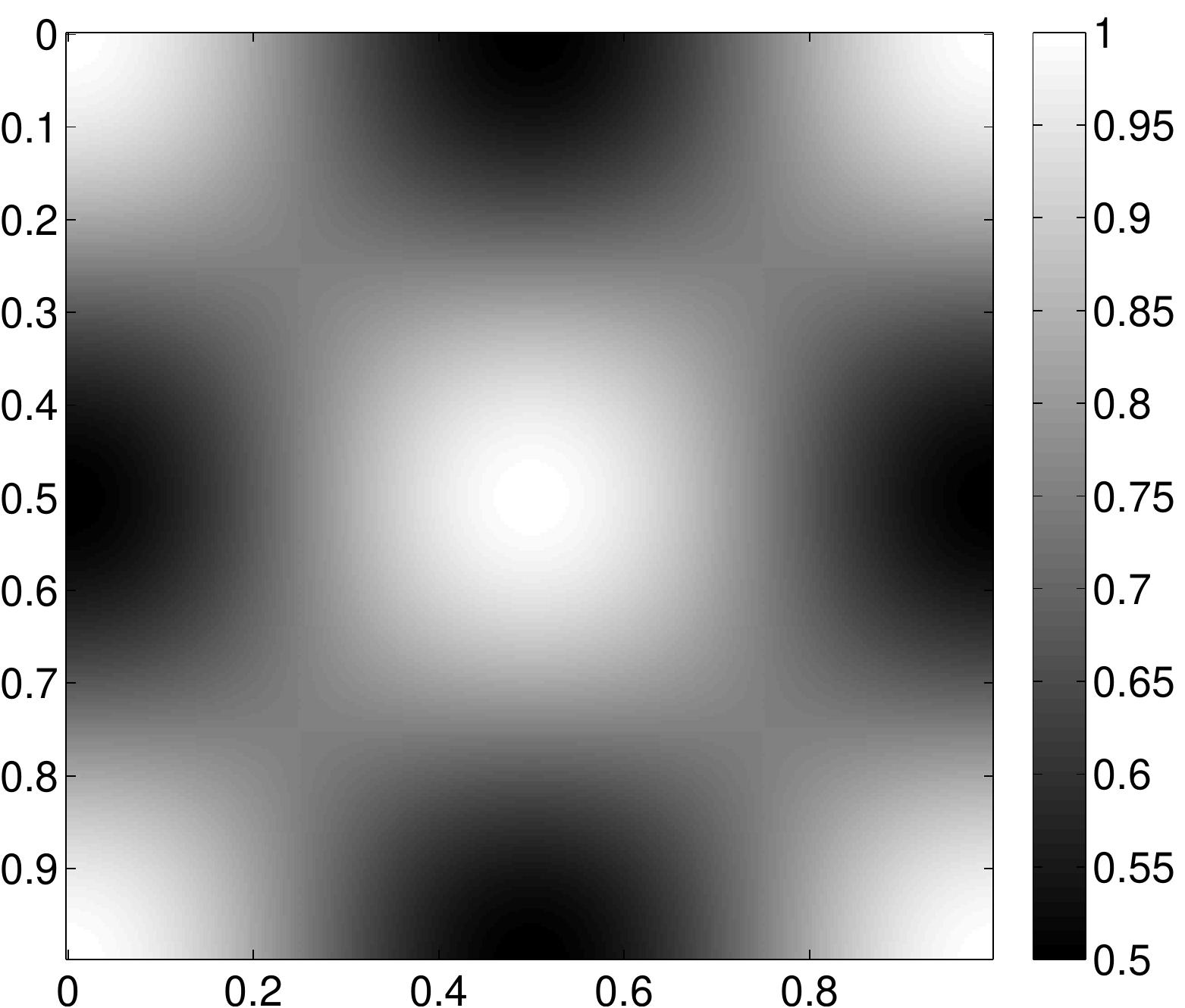} &
      \includegraphics[height=2.5in]{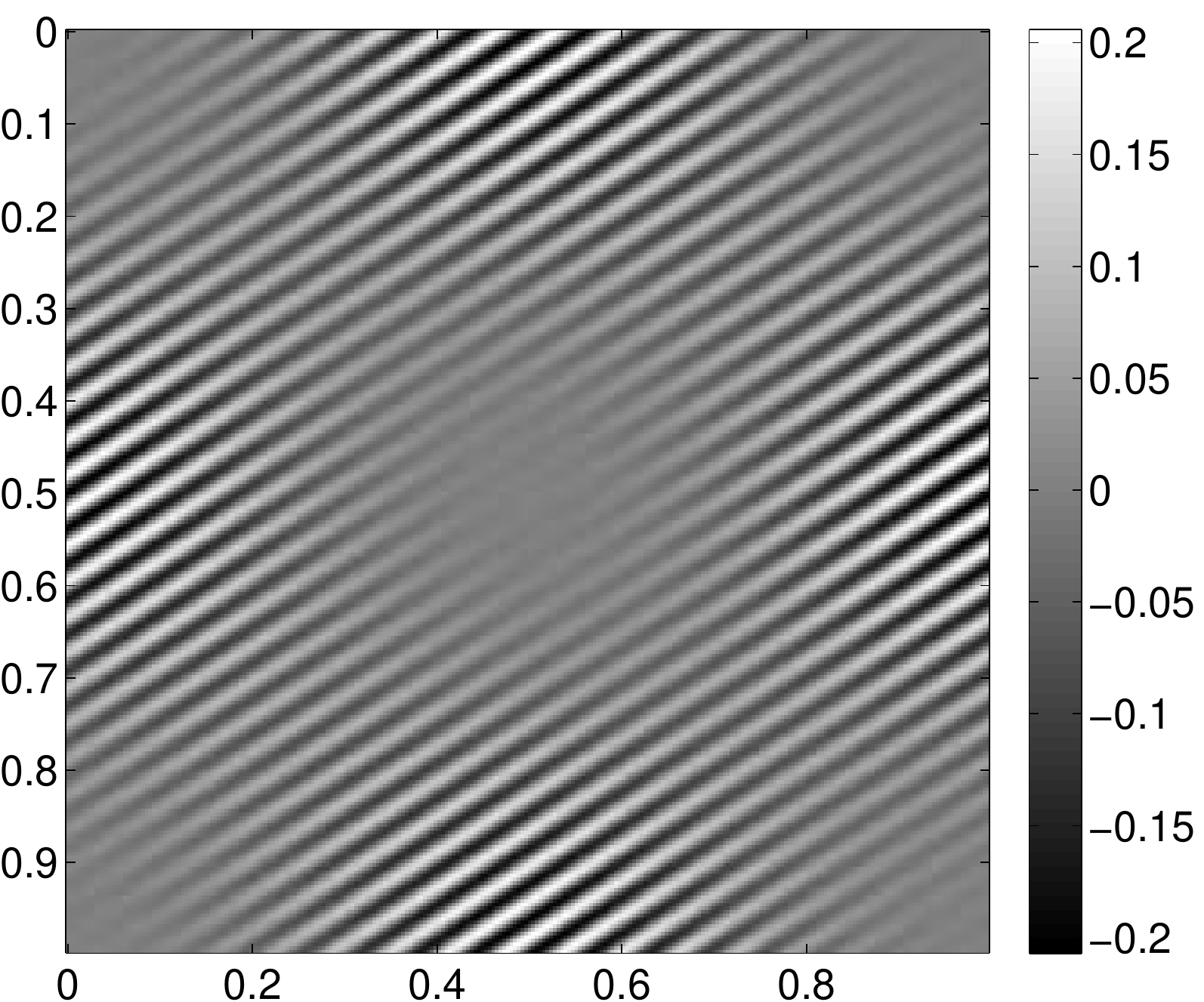}\\
      (a) & (b)
    \end{tabular}
  \end{center}
  \caption{Example 5. The real part of the polarized component $u_+ =
    (u_0 + (iP)^{-1} u_1) / 2$. Notice that the amplitude of $u_+(x)$
    scales with the quantity $\alpha(x)-1$.  $u_- = (u_0 - (iP)^{-1} u_1) /
    2$ is omitted since visually it is close to $u_0$.  }
  \label{fig:po2d_1}
\end{figure}

\paragraph{Example 6.} The coefficient $\alpha(x)$ here is a random
bandlimited function shown in Figure \ref{fig:po2d_2} (a). The initial
conditions are the same as the ones used in Example 5. The real part
of the polarized component $u_+(x)$ is shown in Figures
\ref{fig:po2d_2} (b). Again, we see that the dependence of the
amplitude of $u_+(x)$ on the difference between $\alpha(x)$ and $1$.

\begin{figure}[h!]
  \begin{center}
    \begin{tabular}{cc}
      \includegraphics[height=1.5in]{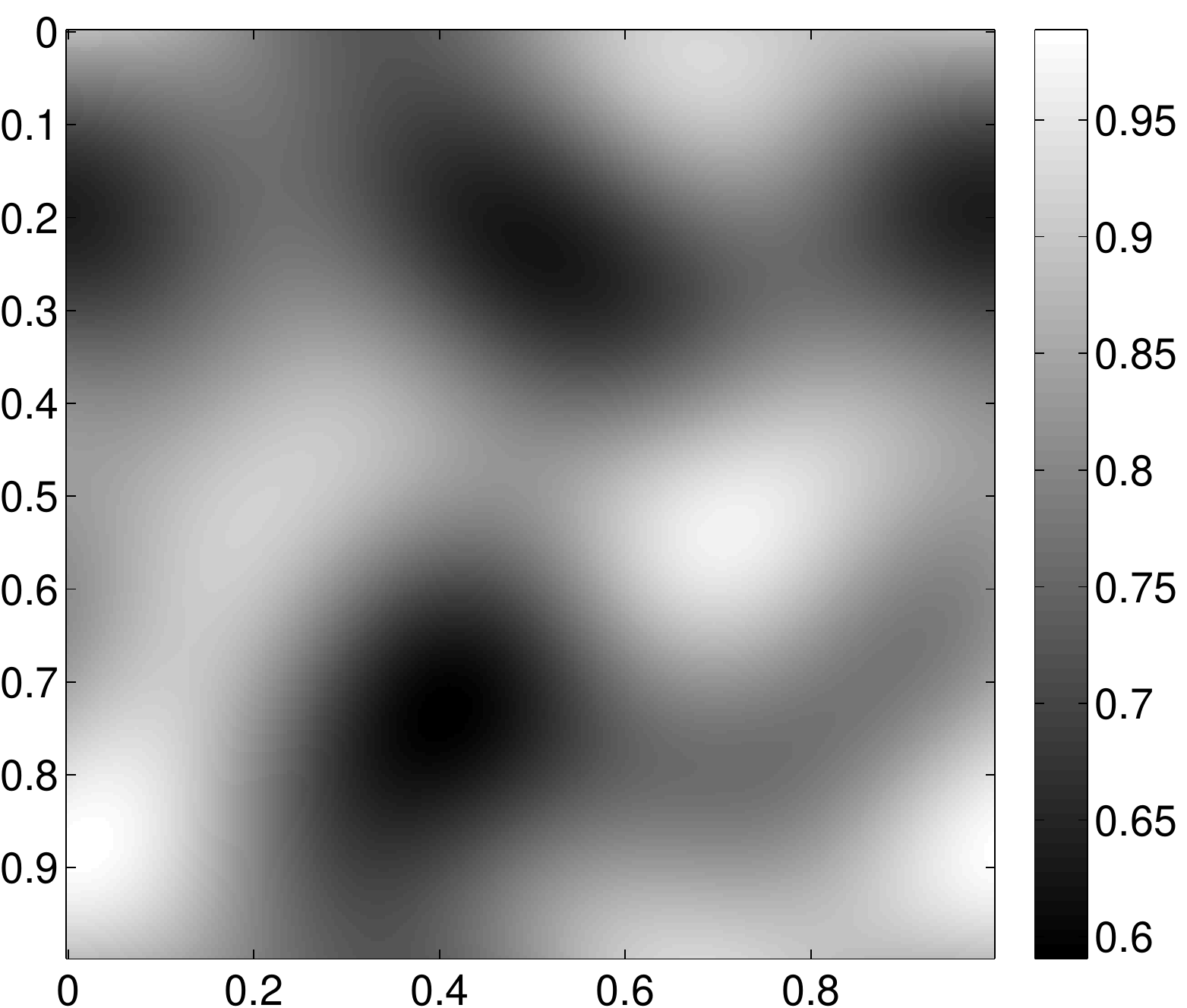} &
      \includegraphics[height=2.5in]{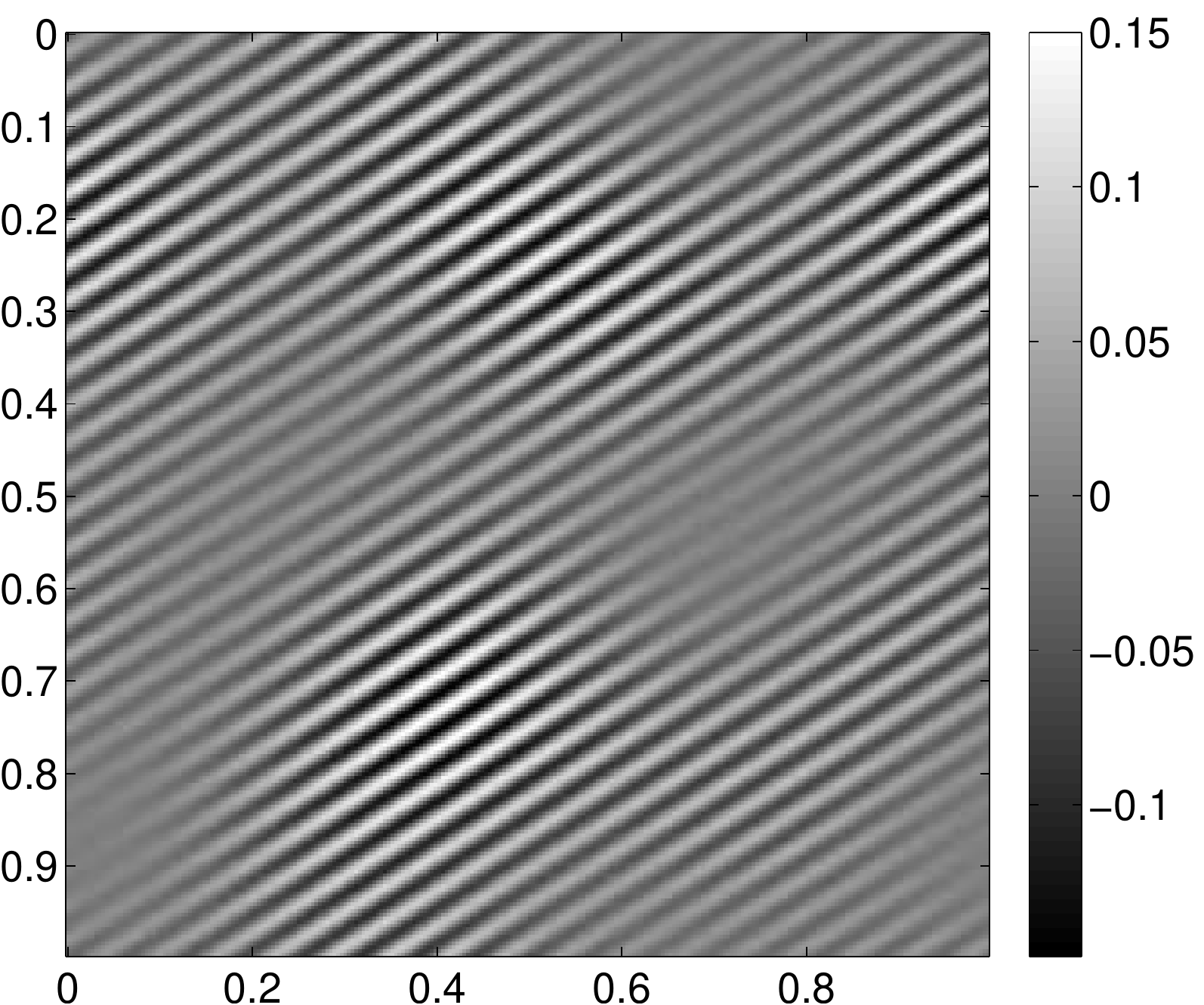}\\
      (a) & (b)
    \end{tabular}
  \end{center}
  \caption{Example 6. The real part of the polarized component $u_+ =
    (u_0 + (iP)^{-1} u_1) / 2$. Notice that the amplitude of $u_+(x)$
    scales with the quantity $\alpha(x)-1$.  $u_- = (u_0 - (iP)^{-1} u_1) /
    2$ is omitted since visually it is close to $u_0$.  }
  \label{fig:po2d_2}
\end{figure}

\subsection{Seismic depth migration}
\label{sec:sdm}
The setup is the same as in the Introduction: consider the Helmholtz equation
\begin{equation}\label{eq:helm2}
  u_{zz} + \Delta_{\bot} + \frac{\omega^2}{c^2(x)} u = 0
\end{equation}
for $z \ge 0$. The transverse variables are either $x \in [0,1]$ in the 1D case, or $x \in [0,1]^2$ in the 2D case. (Our notations support both cases.) Given the wave field $u(x,0)$ at $z=0$, we want to compute the wavefield for $z>0$. For simplicity, we consider periodic boundary conditions in $x$ or $(x,y)$, and no right-hand side in (\ref{eq:helm2}).

As mentioned earlier, we wish to solve the corresponding SSR equation
\begin{equation}\label{eq:depth}
\left( \frac{\pd}{\pd z} - B(z) \right) u = 0,
\end{equation}
where $B(z)$ is a regularized square root of $- \Delta_\bot - \omega^2/c^2(x,z)$. Call $\xi$ the variable(s) dual to $x$. The locus where the symbol $4 \pi^2 |\xi|^2 - \omega^2/c^2(x,z)$ is zero is called the characteristic set of that symbol; it poses well-known difficulties for taking the square root. To make the symbol elliptic (here, negative) we simply introduce
\[
a(z;x,\xi) = g \left( 4 \pi^2 | \xi |^2, \frac{1}{2} \frac{\omega^2}{c^2(x,z)}  \right) - \frac{\omega^2}{c^2(x,z)},
\]
where $g(x,M)$ is a smooth version of the function $\min(x,M)$. Call
$b(z;x,\xi)$ the symbol-square-root of $a(z;x,\xi)$, and $\tilde{B}(z)
= b(z;x,i \nabla_x)$ the resulting operator. A large-frequency cutoff
now needs to be taken to correct for the errors introduced in
modifying the symbol as above. Consider a function $\chi(x)$ equal to
1 in $(-\infty, -2]$, and that tapers off in a $C^\infty$ fashion to
zero inside $[-1, \infty)$. We can now consider $\chi(b(z;x,\xi))$ as
the symbol of a smooth ``directional'' cutoff, defining an operator $X
= \chi(b(z;,x,-i\nabla_x))$ in the standard manner. The operator
$\tilde{B}(z)$ should then be modified as
\[
X \tilde{B}(z) X.
\]
At the level of symbols, this is of course $(\chi(b)) \sharp b \sharp
(\chi(b))$ and should be realized using the composition routine of
discrete symbol calculus.

Once this modified square root has been obtained, it can be used to
solve the SSR equation. It is easy to check that, formally, the
operator mapping $u(x,0)$ to $u(x,z)$ can be written as
\[
E(z) = \exp \int_0^z B(s) \, ds.
\]
If $B(s)$ were to make sense, this formula would be exact. Instead, we
substitute $X \tilde{B}(s) X$ for $B(s)$, and compute $E(z)$ using
discrete symbol calculus. We intend for $z$ to be small, i.e.,
comparable to the wavelength of the field $u(x,0)$, in order to satisfy
a CFL-type condition. With this type of restriction on $z$, the symbol
of $E(z)$ remains sufficiently smooth for the DSC algorithm to be
efficient\footnote{For larger $E(z)$ would be a Fourier integral
  operator, and a phase would be needed in addition to a symbol. We
  leave this to a future project.}: the integral over $s$ can be discretized by
a quadrature over a few points, and the operator exponential can be
realized by scaling-and-squaring as explained earlier.

The effect of the cutoffs $X$ is to smoothly remove 1) turning rays,
i.e, waves that would tend to travel in the horizontal direction or
even overturn, and 2) evanescent waves, i.e., waves that decay
exponentially in $z$ away from $z = 0$. This is why $X$ is called a
directional cutoff. It is important to surround $\tilde{B}$ with
\emph{two} cutoffs to prevent the operator exponential from
introducing energy near the characteristic set of the generating
symbol $4 \pi^2 |\xi|^2 - \omega^2/c^2(x,z)$. This precaution would be
hard to realize without an accurate way of computing compositions
(twisted product). Note that the problem of controlling the frequency
leaking while taking an operator exponential was already addressed by
Chris Stolk in \cite{Stolk2}, and that our approach provides another,
clean solution.

We obtain the following numerical examples.

\paragraph{Example 7.} Let us start by considering the 1D case. The
sound speed $c(x)$ in this example is a Gaussian waveguide (see Figure
\ref{fig:ds1d_1} (a)). We set $\omega$ to be $100 \cdot 2\pi$ in this case.

We perform two tests in this example. In the first test, we select the
boundary condition $u(x,0)$ to be equal to one. This corresponds to
the case of a plane wave entering the waveguide. The solution of
\eqref{eq:depth} is shown in Figure \ref{fig:ds1d_1} (b). As $z$
grows, the wave front starts to deform and the caustics appears at
$x=1/2$ when the sound speed $c(x)$ is minimum.

In the second test of this example, we choose the boundary condition
$u(x,0)$ to be a Gaussian wave packet localized at $x=1/2$. The wave
packet enters the wave guide with an incident angle about $45$
degrees. The solution is shown in Figure \ref{fig:ds1d_1} (c).  Even
though the wave packet deforms its shape as it travels down the wave
guide, it remains localized. Notice that the packet bounces back and
forth at the regions with large sound speed $c(x)$, which is the
result predicted by geometric optics in the high frequency regime.

\begin{figure}[h!]
  \begin{center}
    \begin{tabular}{c}
      \includegraphics[height=1.5in]{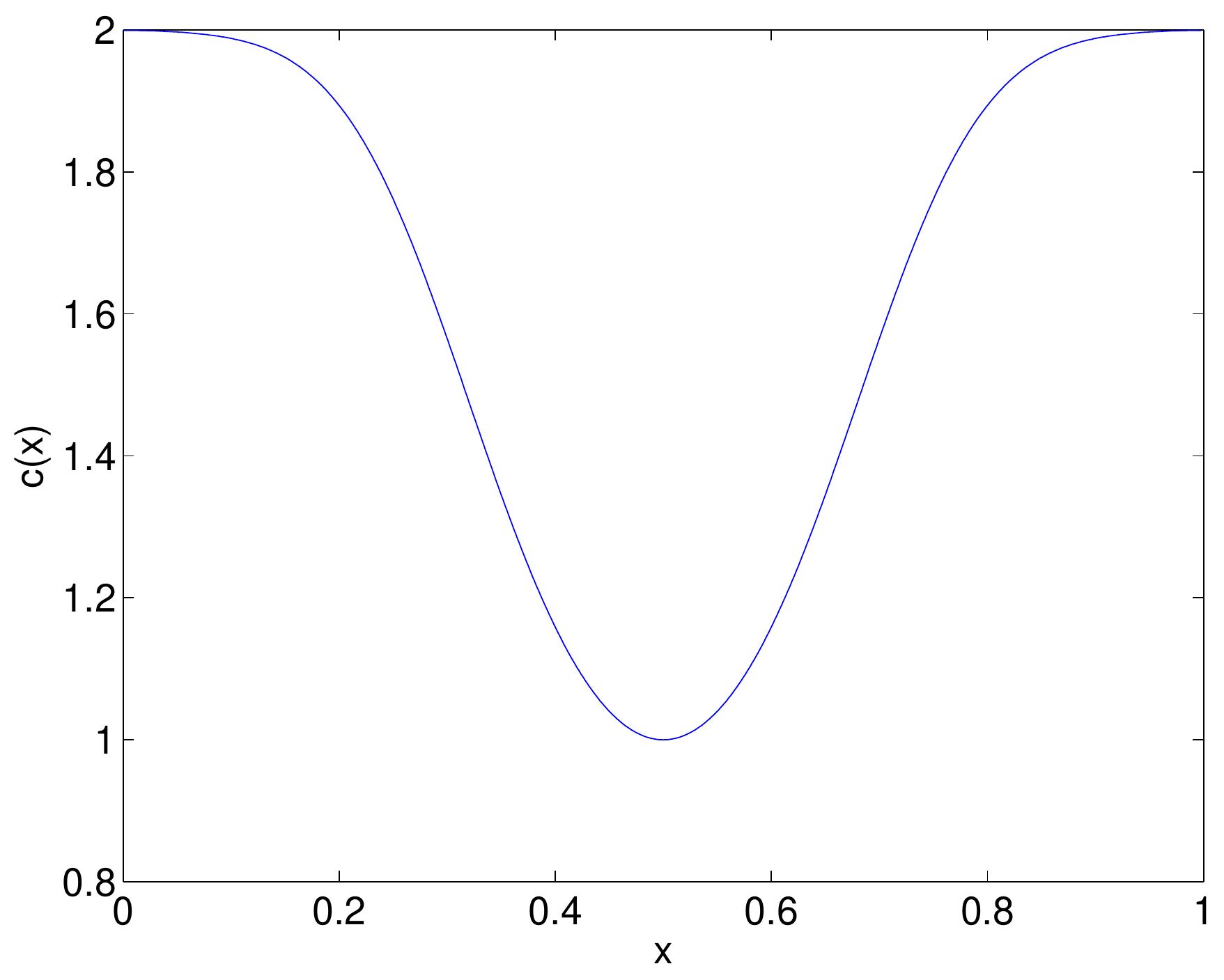}\\
      (a)
    \end{tabular}
    \begin{tabular}{cc}
      \includegraphics[height=2.5in]{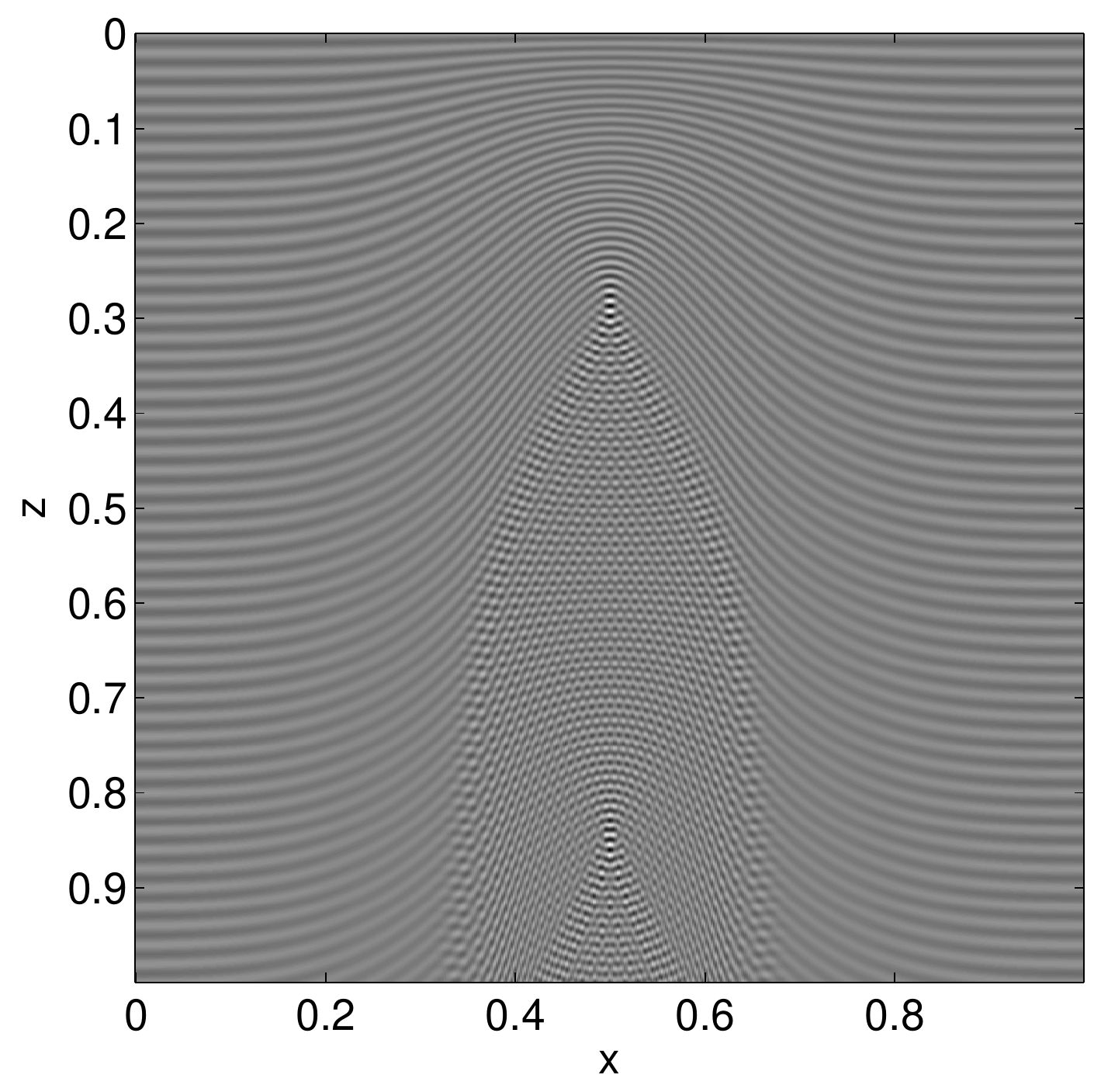} &
      \includegraphics[height=2.5in]{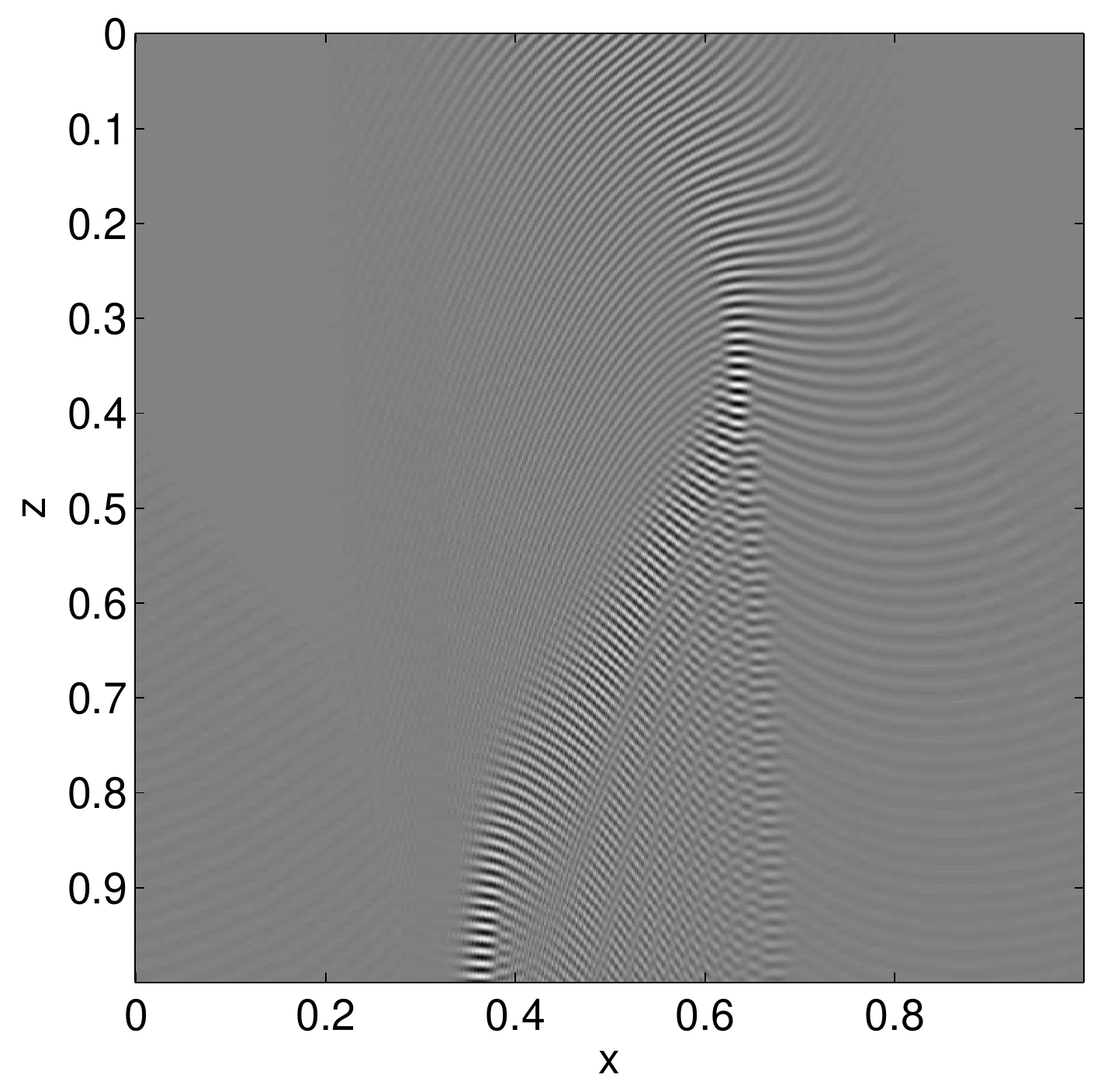}\\
      (b) & (c)
    \end{tabular}
  \end{center}
  \caption{Example 7. (a) sound speed $c(x)$.
    (b) the solution when the boundary condition $u(x,0)$ is a constant.
    (c) the solution when the boundary condition $u(x,0)$ is a wave packet.
  }
  \label{fig:ds1d_1}
\end{figure}

\paragraph{Example 8.} Let us now consider the 2D case. The sound
speed used here is a two dimensional Gaussian waveguide (see Figure
\ref{fig:ds2d_1} (a)). We again perform two different tests. In the
first test, the boundary condition $u(x,y,0)$ is equal to a constant.
The solution at the cross section $y=1/2$ is shown in Figure
\ref{fig:ds2d_1} (b). In the second test, we choose the boundary
condition to be a Gaussian wave packet with oscillation in the $x$
direction. The packet enters the waveguide with an incident angle of
45 degrees. The solution at the cross section $y=1/2$ is shown in
Figure \ref{fig:ds2d_1} (c). Both of these results are similar to the
ones of the one dimensional case.

\begin{figure}[h!]
  \begin{center}
    \begin{tabular}{c}
      \includegraphics[height=1.5in]{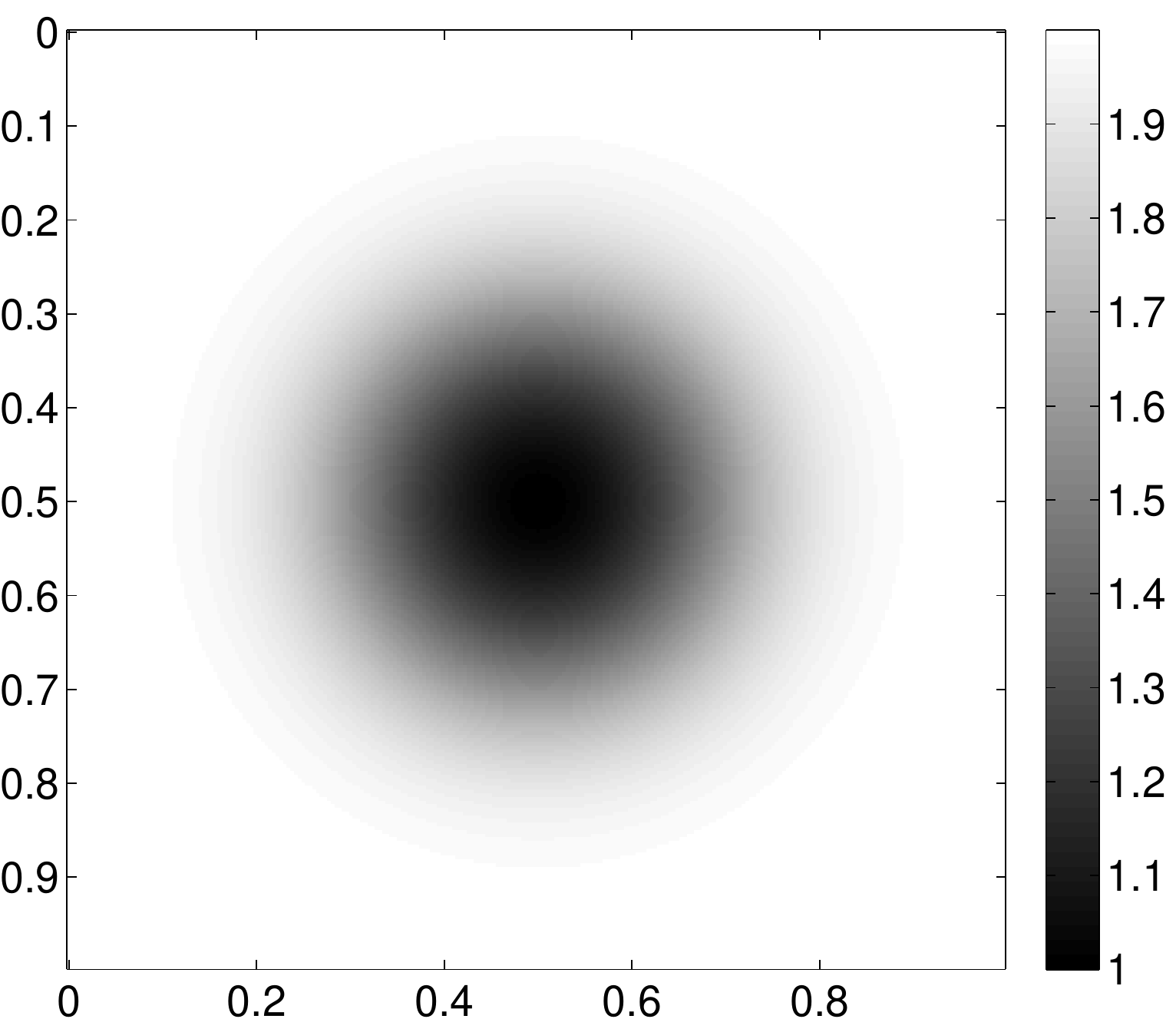}\\
      (a)
    \end{tabular}
    \begin{tabular}{cc}
      \includegraphics[height=2.5in]{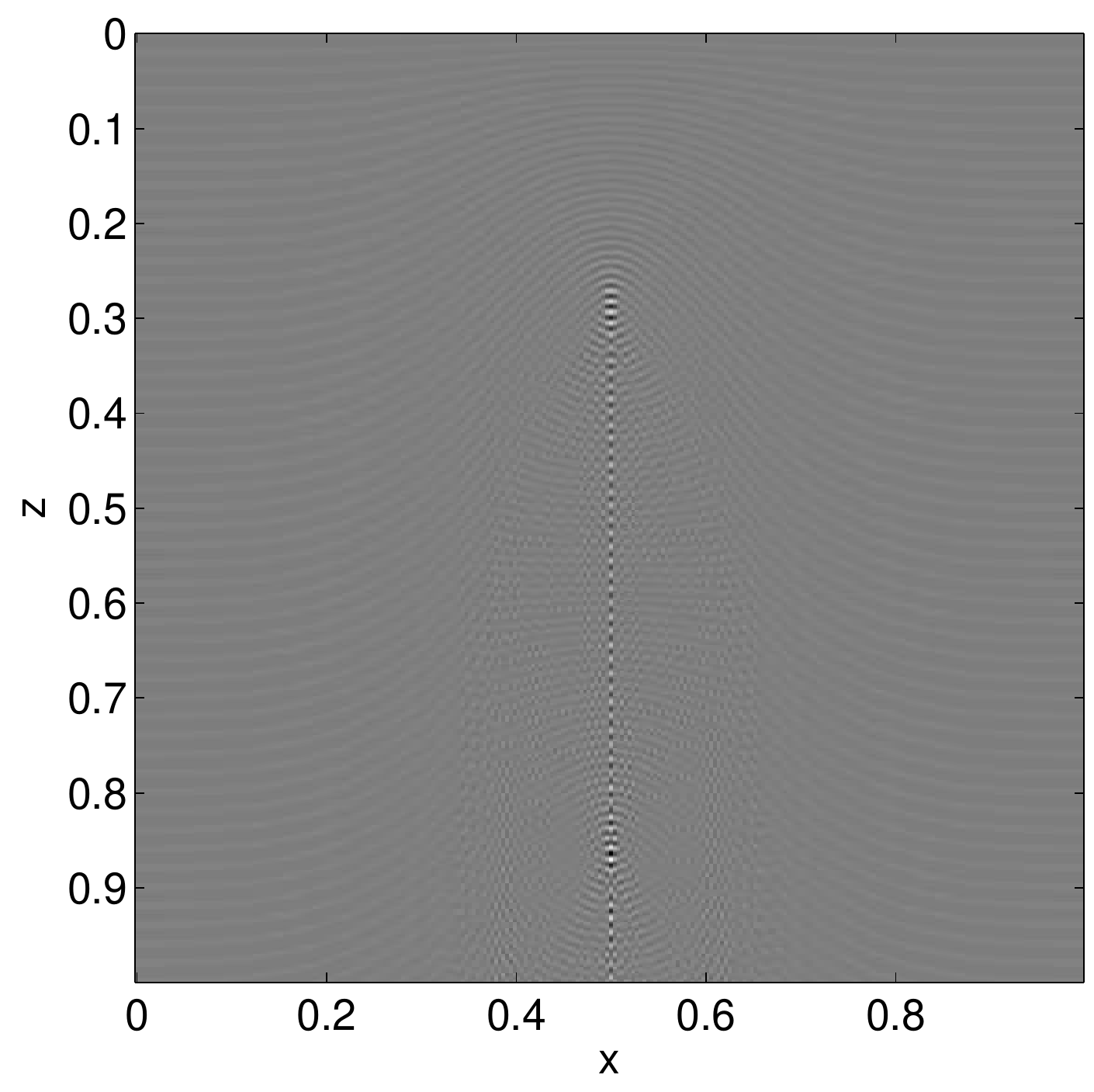} &
      \includegraphics[height=2.5in]{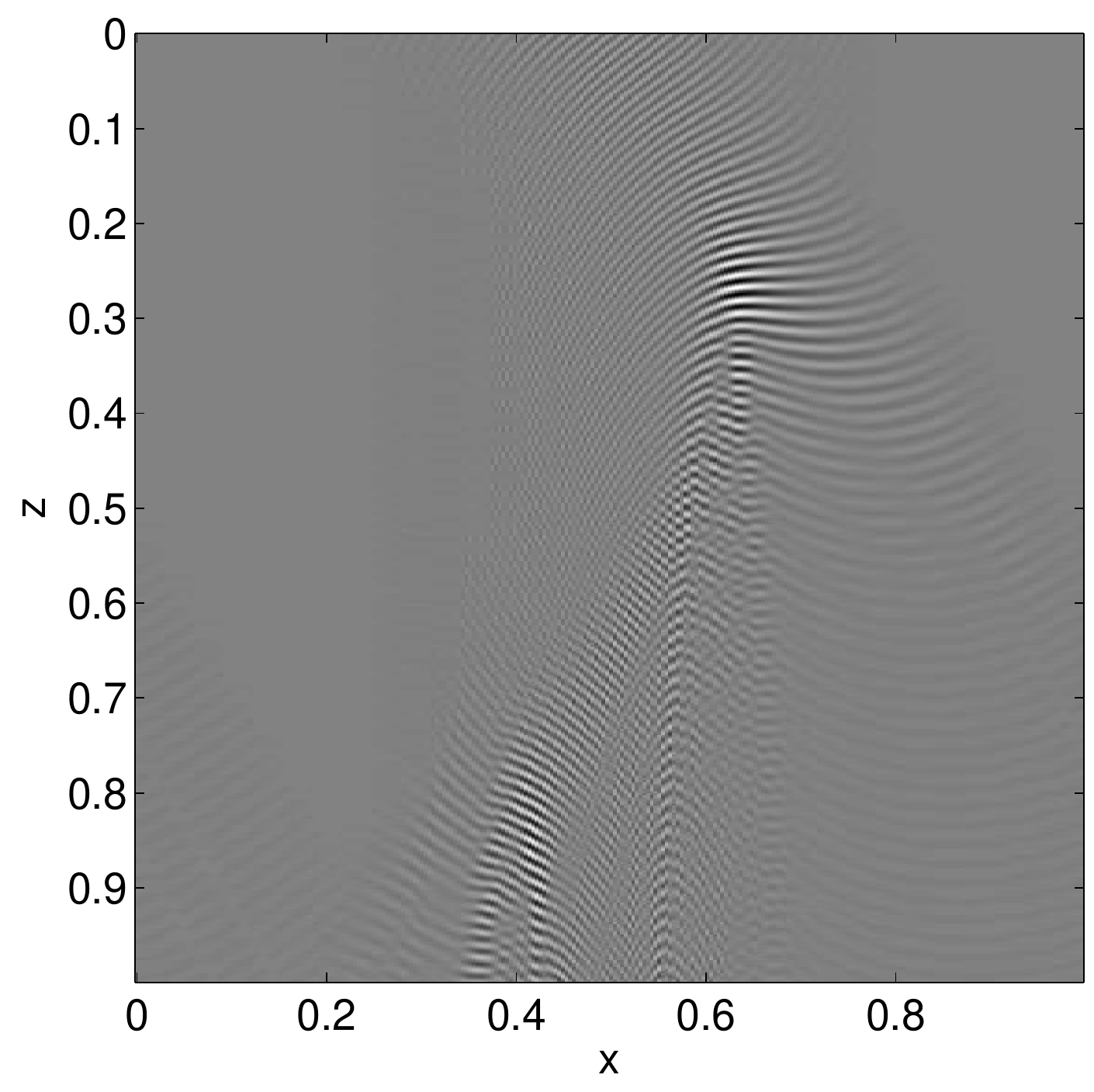}\\
      (b) & (c)
    \end{tabular}
  \end{center}
  \caption{Example 8. (a) sound speed $c(x)$.
    (b) the solution at the cross section $y=1/2$ when the boundary condition $u(x,y,0)$ is a constant.
    (c) the solution at the cross section $y=1/2$ when the boundary condition $u(x,y,0)$ is a wave packet.
  }
  \label{fig:ds2d_1}
\end{figure}


\section{Discussion}\label{sec:diss}




\subsection{Other domains and boundary conditions}

An interesting question is what form discrete symbol calculus should
take when other boundary conditions than periodic are considered, or
on more general domains than a square.

One can speculate that the discrete Sine transform (DST) should be
used as $e_\lambda$ for Dirichlet boundary conditions on a rectangle,
or the discrete Cosine transform (DCT) for Neumann on a rectangle.
Whatever choice is made for $e_\lambda$ should dictate the definition
of the corresponding frequency variable $\xi$. A more robust approach
could be to use spectral elements for more complicated domains, where
the spectral domain would be defined by Chebyshev expansions. One may
also imagine expansions in prolate spheroidal wavefunctions.
Regardless of the type of expansions chosen, the theory of
pseudodifferential operators on bounded domains is a difficult topic
that will need to be understood.

Another interesting problem is that of designing absorbing boundary
conditions in variable media. We hope that the ideas of symbol
expansions will provide new insights for this question.


\subsection{Other equations}

Symbol-based methods may help solve other equations than elliptic PDE.
The heat equation in variable media comes to mind: its fundamental
solution has a nice pseudodifferential smoothing form that can be
computed via scaling-and-squaring.

A more challenging example are hyperbolic systems in variable, smooth
media. The time-dependent Green's function of such systems is not a
pseudodifferential operator, but rather a Fourier integral operator
(FIO), where $e^{2 \pi i x\cdot \xi} a(x,\xi)$ needs to be replaced by
$e^{\Phi(x,\xi)} a(x,\xi)$. We regard the extension of discrete symbol
calculus to handle such phases a very interesting problem, see
\cite{FastFIO, OptimFIO} for preliminary results on fast application of FIO.

\appendix

\section{Appendix}

\begin{proof}[Proof of Lemma \ref{teo:L2bdd}.]
As previously, write
\[
\hat{a}_\lambda(\xi) = \int e^{-2 \pi i x \cdot \lambda} a(x,\xi) \, dx
\]
for the Fourier series coefficients of $a(x,\xi)$ in $x$. Then we can express (\ref{eq:psido-sum}) as
\[
(A f)(x) = \sum_{\xi \in \Z^d}  e^{2 \pi i x \cdot \xi} \sum_{\lambda \in \Z^d} e^{2 \pi i x \cdot \lambda}  \hat{a}_\lambda(\xi) \hat{f}(\xi).
\]
We seek to interchange the two sums. Since $a(x,\xi)$ is differentiable $d'$ times, we have
\[
(1 + | 2 \pi \lambda |^{d'} ) \, \hat{a}_\lambda(\xi) = \int_{[0,1]^d} e^{-2 \pi i x \cdot \lambda} (1 + (- \Delta_x)^{d'/2}) a(x,\xi) \, dx,
\]
hence $|\hat{a}_\lambda(\xi)| \leq (1 + | 2 \pi \lambda |^{d'} )^{-1} \| (1 + (- \Delta_x)^{d'/2}) a(x,\xi) \|_{L^\infty_x}$. The exponent $d'$ is chosen so that $\hat{a}_\lambda(\xi)$ is absolutely summable in $\lambda \in \Z^d$. If in addition we assume $\hat{f} \in \ell_1(\Z^d)$, then we can apply Fubini's theorem and write
\[
(A f)(x) = \sum_{\lambda \in \Z^d} A^\lambda f(x),
\]
where $A^\lambda f(x) = e^{2 \pi i x \cdot \lambda} (M_{\hat{a}_\lambda(\xi)} f)(x)$, and $M_g$ is the operator of multiplication by $g$ on the $\xi$ side. By Plancherel, we have
\[
\| A^\lambda f \|_{L^2} = \| M_{\hat{a}_\lambda(\xi)} f \|_{L^2} \leq \sup_{\xi} |\hat{a}_\lambda(\xi)| \cdot \| f \|_{L^2}.
\]
Therefore, by the triangle inequality,
\begin{align*}
\| A f \|_{L^2} &\leq \sum_{\lambda \in \Z^d} \| A^\lambda f \|_{L^2} \\
&\leq \sum_{\lambda \in \Z^d} (1 + | 2 \pi \lambda |^{d'} )^{-1} \cdot \sup_{x,\xi} | (1 + (- \Delta_x)^{d'/2}) a(x,\xi) | \cdot \| f \|_{L^2}
\end{align*}
As we have seen, the sum over $\lambda$ converges. This proves the theorem when $f$ is sufficiently smooth; a classical density argument shows that the same conclusion holds for all $f \in L^2([0,1]^d)$.

\end{proof}

\end{document}